\documentclass[12pt]{article} 
\usepackage{amsfonts,amsmath,latexsym,amssymb,mathrsfs,amsthm}

\def\numberlikeadb{\global\def\theequation{\thesection.\arabic{equation}}}
\numberlikeadb
\newtheorem{theorem}{Theorem}[section]
\newtheorem{lemma}[theorem]{Lemma}
\newtheorem{corollary}[theorem]{Corollary}
\newtheorem{definition}[theorem]{Definition}
\newtheorem{proposition}[theorem]{Proposition}
\newtheorem{remark}[theorem]{Remark}

\numberwithin{equation}{section}

\begin{document}

\title{On Stein's method for products of normal random variables and zero bias couplings}
\author{Robert E. Gaunt\footnote{Department of Statistics,
University of Oxford, 24--29 St$.$ Giles', Oxford OX1 3LB, UK }
}

\date{March 2016} 
\maketitle

\begin{abstract}In this paper we extend Stein's method to the distribution of the product of $n$ independent mean zero normal random variables.  A Stein equation is obtained for this class of distributions, which reduces to the classical normal Stein equation in the case $n=1$.  This Stein equation motivates a generalisation of the zero bias transformation.  We establish properties of this new transformation, and illustrate how they may be used together with the Stein equation to assess distributional distances for statistics that are asymptotically distributed as the product of independent central normal random variables.  We end by proving some product normal approximation theorems.
\end{abstract}


\noindent{{{\bf{AMS 2010 Subject Classification:}}} Primary 60F05, 60E10

\section{Introduction}In 1972, Stein \cite{stein} introduced a powerful method for deriving bounds for normal approximation.  Since then, Stein's method has been extended to many other distributions, such as the Poisson \cite{chen 0}, beta \cite{dobler1}, \cite{goldstein4}, gamma \cite{gaunt chi square}, \cite{luk}, \cite{nourdin1}, exponential \cite{chatterjee}, \cite{pekoz1}, Laplace \cite{dobler}, \cite{pike} and variance-gamma \cite{eichelsbacher}, \cite{gaunt}; for an overview see Reinert \cite{reinert 0}.  

Stein's method for normal approximation rests on the following characterization of the normal distribution, which can be found in Stein \cite{stein2}, namely $Z\sim N(0,\sigma^2)$ if and only if 
\begin{equation} \label{stein lemma}\mathbb{E}[\sigma^2f'(Z)-Zf(Z)]=0
\end{equation}
for all real-valued absolutely continuous functions $f$ such that $\mathbb{E}|f'(Z)|$ exists.  This gives rise to the following inhomogeneous differential equation, known as the Stein equation:
\begin{equation} \label{normal equation} \sigma^2f'(x)-xf(x)=h(x)-\mathbb{E}h(Z),
\end{equation}
where $Z \sim N(0,\sigma^2)$, and the test function $h$ is real-valued.  For any real-valued bounded measurable test function $h$, a solution $f$ to (\ref{normal equation}) exists.  Now, evaluating both sides at any random variable $W$ and taking expectations gives
\begin{equation} \label{expect} \mathbb{E}[\sigma^2f'(W)-Wf(W)]=\mathbb{E}h(W)-\mathbb{E}h(Z).
\end{equation}
Thus, the problem of bounding the quantity $\mathbb{E}h(W)-\mathbb{E}h(Z)$ reduces to the bounding the left-hand side of (\ref{expect}). 

\subsection{The product normal distribution}

In this paper, we extend Stein's method to products of independent central normal random variables.  The probability density function of the products of independent mean zero normal variables was shown by Springer and Thompson \cite{springer} to be a Meijer $G$-function (defined in Appendix B.1.1.).  The probability density function of the product $Z=X_1X_2\cdots X_n$ of
independent normal random variables $N(0, \sigma_{i}^2)$, $i= 1,2,..., n$, is given by
\begin{equation}\label{MeijerN} p_n(x)=\frac{1}{(2\pi)^{n/2}\sigma}G_{0,n}^{n,0}\bigg(\frac{x^2}{2^n\sigma^2} \; \bigg| \;0,\ldots,0\bigg), \quad x\in\mathbb{R},
\end{equation}
where $\sigma=\sigma_{1}\sigma_{2}\cdots\sigma_{n}$.  If (\ref{MeijerN}) holds, then we say that $Z$ has a \emph{product normal} distribution, and write $Z\sim \mathrm{PN}(n,\sigma^2)$.  For the case of two products, (\ref{MeijerN}) simplifies to
\[p_2(x)=\frac{1}{\pi\sigma_{1}\sigma_{2}}K_0\bigg(\frac{|x|}{\sigma_{1}\sigma_{2}}\bigg), \quad x\in\mathbb{R},\]
where $K_0(x)$ is a modified Bessel function of the second kind (defined in Appendix B.2.1).  

The density $p_n$ satisfies the asymptotic formulas
\begin{equation*}p_n(x)\sim\frac{1}{(2\pi)^{n/2}(n-1)!\sigma}\bigg[-\log\bigg(\frac{|x|}{\sigma}\bigg)\bigg]^{n-1}, \quad \mbox{as\:}x\rightarrow0,
\end{equation*}
and
\begin{equation}\label{pnxinfty}p_n(x)\sim \frac{2^{n/2-1}}{\sigma\sqrt{\pi n}}\bigg(\frac{|x|}{\sigma}\bigg)^{1/n-1}\exp\bigg(-\frac{n}{2}\bigg(\frac{|x|}{\sigma}\bigg)^{2/n}\bigg), \quad \mbox{as\:}|x|\rightarrow\infty
\end{equation}
(see Theorem 6 of Springer and Thompson \cite{springer}, and Luke \cite{luke}, Section 5.7, Theorem 5).  It is interesting to compare these asymptotic formulas with the exact formula for the density of the random variable $|X|^n\mathrm{sgn}(X)$, where $X\sim N(0,\sigma^2)$, as given by
\begin{equation*}p(x)=\frac{1}{n\sigma\sqrt{2\pi}}\bigg(\frac{|x|}{\sigma}\bigg)^{1/n-1}\exp\bigg(-\frac{1}{2}\bigg(\frac{|x|}{\sigma}\bigg)^{2/n}\bigg), \quad x\in\mathbb{R},
\end{equation*}
which is seen to have a greater singularity at the origin and slower decay at the tails than the $\mathrm{PN}(n,\sigma^2)$ distribution.

\subsection{A Stein equation for the product normal distribution}

One of the main results of this paper is a Stein equation for the $\mathrm{PN}(n,\sigma^2)$ distribution:
\begin{equation} \label{delta009} \frac{\sigma^2}{x}\bigg(x\frac{\mathrm{d}}{\mathrm{d}x}\bigg)^nf(x)-xf(x)=h(x)-\mathrm{PN}_n^{\sigma^2}h,
\end{equation}
where $\mathrm{PN}_n^{\sigma^2}h$ denotes the quantity $\mathbb{E}h(X)$ for $X\sim \mathrm{PN}(n,\sigma^2)$.  For ease of reading, we define the operators $Tf(x)=xf'(x)$ and $A_nf(x)=x^{-1}T^nf(x)$.  With this notation, the $\mathrm{PN}(n,\sigma^2)$ Stein equation (\ref{delta009}) can be written as
\begin{equation}\label{delta}\sigma^2A_nf(x)-xf(x)=h(x)-\mathrm{PN}_n^{\sigma^2}h.
\end{equation}
For all $n\geq1$, the operator $T^n$ satisfies the fundamental identity (Luke \cite{luke}, p$.$ 24) $T^nf(x)=\sum_{k=1}^n{n\brace k}x^kf^{(k)}(x)$, where ${n\brace k}=\frac{1}{k!}\sum_{j=0}^k(-1)^{k-j}\binom{k}{j}j^n$ are Stirling numbers of the second kind (Olver et al$.$ \cite{olver}, Chapter 26).  Consequently, (\ref{delta}) is a $n$-th order linear
differential equation with simple coefficients:
\begin{equation}\label{delta789}\sigma^2\sum_{k=1}^n{n\brace k}x^{k-1}f^{(k)}(x)-xf(x)=h(x)-\mathrm{PN}_n^{\sigma^2}h.
\end{equation}
As an $n$-th order differental equation, this Stein equation is almost unique in the literature.  The exceptions being the $n$-th order Stein equations of Goldstein and Reinert \cite{goldstein3} that involve orthogonal polynomials, and the very recent $n$-th order Stein equations for products of $n$ independent beta and gamma random variables (see Gaunt \cite{gaunt ngb}) and linear combinations of $n$ independent gamma random variables (see Arras et al$.$ \cite{aaps16}).

The $\mathrm{PN}(n,\sigma^2)$ Stein equation is a natural generalisation of the normal Stein equation (\ref{normal equation}) to one for products of independent central normal random variables and has a number of attractive properties; a further discussion is given in Remark \ref{nice eqn}.  For the case $n=2$, (\ref{delta789}) reduces to
\begin{equation}\label{deltappp}\sigma^2xf''(x)+\sigma^2f'(x)-xf(x)=h(x)-\mathrm{PN}_2^{\sigma^2}h,
\end{equation}
and in Lemma \ref{forty} we obtain the unique bounded solution of (\ref{deltappp}), as well as bounds on its first four derivatives (see Theorem \ref{compoz7}), which hold provided that $h$ and its first three derivatives are bounded.  Note that (\ref{deltappp}) is a second order linear differential equation involving $f$, $f'$ and $f''$.  Such Stein equations are uncommon in the literature, although Pek\"oz et al$.$ \cite{pekoz} have recently obtained a similar Stein equation for the family of densities
\begin{equation*}\label{kummeru}\kappa_s(x)=\Gamma(s)\sqrt{\frac{2}{s\pi}}\exp\bigg(-\frac{x^2}{2s}\bigg)U\bigg(s-1,\frac{1}{2},\frac{x^2}{2s}\bigg), \quad x>0,\: s\geq \frac{1}{2},
\end{equation*}
where $U(a,b,x)$ denotes the confluent hypergeometric function of the second kind (see Olver et al$.$ \cite{olver}, Chapter 13).  Pike and Ren \cite{pike} have also obtained a second order Stein operator for the Laplace distribution.  It should also be noted that the $\mathrm{PN}(2,\sigma^2)$ distribution is a member of the class of variance-gamma distributions, and that (\ref{deltappp}) is indeed a special case of the variance-gamma Stein equation, recently introduced by Gaunt \cite{gaunt vg}.  Very recently, Gaunt \cite{gaunt gh} obtained a second order Stein operator for the generalized hyperbolic distribution.  A limiting case of this Stein operator is the variance-gamma Stein operator. 

In Section 2.3, we solve the product normal Stein equation.  For the case of two products, we obtain uniform bounds for the solution and its first four derivatives in terms of supremum norms of derivatives of the test function $h$.  However, for $n\geq3$, the solution takes on a rather complicated form, involving integrals of Meijer $G$-functions, and we have been unable to obtain estimates for the solution and its derivatives.  This is left as an interesting open problem and is discussed further in Section 2.3.2.  

\subsection{A generalised zero bias transformation}
 
Returning to normal approximation, the quantity $\mathbb{E}[\sigma^2f'(W)-Wf(W)]$ is typically bounded through the use of a coupling and Taylor expansions.  A coupling technique that is commonly used in the case mean zero random variables is the zero bias transformation, which was introduced by Goldstein and Reinert \cite{goldstein}.  If $W$ is a mean zero random variable with finite, non-zero variance $\sigma^2$, we say that $W^*$ has the $W$-zero biased distribution if for all differentiable $f$ for which $\mathbb{E}Wf(W)$ exists,
\begin{equation*}\mathbb{E}Wf(W)=\sigma^2\mathbb{E}f'(W^*).
\end{equation*}
The above definition shows why we might like to use a zero-biasing method for normal approximation: it gives a way of splitting apart an expectation, and reduces normal approximation to bounding the quantity $\sigma^2E[f'(W)-
f'(W^*)]$.  We therefore have
\begin{equation}\label{zero eqn 1}\mathbb{E}[\sigma^2f'(W)-Wf(W)]=\sigma^2\mathbb{E}[f'(W)-f'(W^*)],
\end{equation}
and the right-hand side may be bounded by Taylor expanding about $W$.  Goldstein and Reinert \cite{goldstein} presented a number of interesting properties and obtained some useful constructions, such as:  
\begin{lemma}\label{construct2}Let $X_1,\ldots,X_n$ be independent mean zero random variables with $\mathbb{E}X_i^2=\sigma_i^2$.  Set $W=\sum_{i=1}^nX_i$ and $\mathbb{E}W^2=\sigma^2$.  Let $I$ be a random index independent
of the $X_i$ such that $\mathbb{P}(I=i)=\frac{\sigma_i^2}{\sigma^2}.$  Let
\[W_i=W-X_i=\sum_{j\not= i}X_j.\]
Then $W_I+X_I^*$ has the $W$-zero biased distribution, where $X_I^*$ has the $X_I$-zero biased distribution.
\end{lemma}
Such constructions combined with a Taylor expansion of the right-hand side of (\ref{zero eqn 1}) often allow simple proofs of limit theorems for normal approximation.  

Motivated by the zero bias transformation and the multivariate normal Stein equation, Goldstein and Reinert \cite{goldstein 2} extended the concept of the zero bias transformation to any finite dimension.  In this paper we introduce another generalisation of the zero bias transformation.  The product normal Stein equation and zero bias transformation motivate the following definition.
\begin{definition} \label{nthzero} Let $W$ be a mean zero random variable with finite, non-zero variance $\sigma^2$. We say that $W^{*(n)}$ has the \emph{$W$-zero biased distribution of order $n$} if for all $n$ times differentiable $f$ for which $\mathbb{E}Wf(W)$ exists,
\begin{equation}\label{islip}\mathbb{E}Wf(W)=\sigma^2\mathbb{E}A_nf(W^{*(n)}),
\end{equation}
where $A_nf(x)=x^{-1}T^nf(x)$ and $Tf(x)=xf'(x)$.
\end{definition} 
The existence of the zero biased distribution of order $n$ for any $W$ with zero mean and non-zero, finite variance is established by Lemma \ref{etna}.  The zero bias transformation of order $n$ is a natural generalisation of the zero bias transformation to the study of products of independent normal distributions in the same way that the multivariate zero bias transformation, introduced by Goldstein and Reinert \cite{goldstein 2}, is a natural extension to random vectors in $\mathbb{R}^d$. 

The zero bias transformation of order $n$ has a number of interesting properties that generalise those of the zero bias transformation.  These properties are collected in Propositions \ref{doublesquare} and \ref{john hates}.  The application of the zero bias transformation of order $n$ to assess distributional distances for statistics that have a limiting product normal distribution is analogous to that of the zero bias transformation for normal approximation.  We illustrate how such approximation results can be obtained with Theorem \ref{jazzz}, provided that we have bounds on the relevant derivatives of the solution to the Stein equation (\ref{delta}). 

In obtaining these bounds, we use the fact (Proposition \ref{john hates}, part (iv)) that the zero bias transformation of order $n$ of the random variable $W=\prod_{k=1}^nW_k$, where the $W_k$ are independent, is given by $W^{*(n)}=\prod_{k=1}^nW_k^*$.  One can then construct the zero bias transformations for each of the $W_k$ by using one of the constructions given by Goldstein and Reinert \cite{goldstein}.  For example, if $W_k=\frac{1}{\sqrt{n}}\sum_{i=1}^nX_{ik}$, where the $X_{ik}$ are independent random variables with mean zero and non-zero, finite variance, then we can construct $W_k^*$ using Lemma \ref{construct2}. 

It is worth comparing this construction with an analogous result, due to Luk \cite{luk}, involving size bias couplings (for an application of this coupling to normal approximation see Baldi, Rinott and Stein \cite{baldi}).  If $W\geq0$ has mean $\mu>0$, we say $W^s$ has the $W$-size biased distribution if for all $f$ such that $\mathbb{E}Wf(W)$ exists,
\[\mathbb{E}Wf(W)=\mu\mathbb{E}f(W^s).\] 
Now, if $W=\prod_{k=1}^nW_k$, where the $W_k$ are positive, independent random variables and $W_1^s,\ldots,W_n^s$ are independent random variables with $W_k^s$ having the $W_k$-size biased distribution, then $W^s=\prod_{k=1}^nW_k^s$ has the $W$-size biased distribution.  The constructions are similar, although the product of $n$ zero bias distributions has the $W$-zero bias distribution of order $n$, rather than the $W$-zero bias distribution.

In Section 4, we illustrate how the product normal Stein equation may be used together with the zero bias transformation of order $n$ to assess distributional distances for statistics that are asymptotically distributed as products of independent normal random variables, and we obtain explicit bounds on the convergence rate in the $n=2$ case.  We are restricted to the case $n=2$ because we are unable to obtain bounds for the solution of $\mathrm{PN}(n,\sigma^2)$ Stein equation for $n\geq3$.  In Section 5, we get around this difficulty by using a recent results, due to Gaunt \cite{gaunt}, that allow one to obtain bounds on distributional distances when the limit distribution that be represented as a function of a multivariate normal random variable, of which the product normal distribution is an obvious example.  We end Section 5 by comparing the bounds obtained by this approach that bypasses the Stein equation and those derived in Section 4 using the Stein equation.

The approach taken in this paper is somewhat classical, although it should be noted that product normal limit theorems have recently been established in the context of Malliavin calculus.  Indeed, the Malliavin-Stein method (see Nourdin and Peccati \cite{np12} for an introduction), that was originally developed for Gaussian approximation, is applicable to a wide class of laws; see, for example, Eden and Viens \cite {eden1} and Eden and Viquez \cite{eden2}.  Eichelsbacher and Th\"{a}le \cite{eichelsbacher}  used the Malliavin-Stein approach  to obtain bounds for variance-gamma approximation of functionals of isonormal Guassian processes.  In particular, they showed that a sequence of random variables in the second Wiener chaos converges to a variance-gamma limit if and only if their second to sixth moments converge to those a variance-gamma random variable.  A special case of this general six moment theorem is a six moment theorem for the distribution of the product of two (possibly correlated) standard normal variables (see \cite{eichelsbacher}, Theorem 5.8 and Remark 5.9), with explicit bounds on the rate of convergence.

The results of \cite{eichelsbacher} are complemented by those of Azmooden et al$.$ \cite{azmooden}.  They established necessary and sufficient conditions under which a sequence of random variables living inside a Wiener chaos of arbitrary order converge to limiting random variables, whose distribution is of the form $\sum_{i=1}^k\alpha_i(Z_i^2-1)$, where $k$ is a finite integer, the $\alpha_i$, $i=1,\ldots,k$, are piecewise distinct and the $Z_i$ are independent $N(0,1)$ variables.  The case $k=2$, $\alpha_1=\frac{1}{2}=-\alpha_2$ corresponds to a limit with the same distribution as a $\mathrm{PN}(2,1)$ random variable, and other parameter values can yield results for limiting distributions that fall outside the variance-gamma class. 

\subsection{Outline of the paper}

We begin Section 2 by obtaining some useful properties for the operator $A_n$.  We use some of these properties to prove the existence and uniqueness of the zero bias distribution of order $n$.  We then present a characterising equation for the product normal distribution which motivates the Stein equation (\ref{delta}).  Also, for the case $n=2$ we obtain the unique bounded solution, as well as bounds on its first four derivatives.  The general case $n\geq3$ is more difficult, but we are still able to solve the equation.  We also discuss how one may extend the approach used in the $n=2$ case to this more challenging case.  In this section, we also consider an application of our product normal characterisation to the problem of obtaining a formula for the characteristic function of a product normal distributed random variable.  In Section 3, we present some of the properties of the zero bias distribution of order $n$.  In Section 4, we illustrate how the product normal Stein equation may be used together with the zero bias transformation to derive product normal approximation results.  In Section 5, we bypass the product normal Stein equation to prove two product normal approximations for general $n$.  The proof of Theorem \ref{compoz7} is given in Appendix A.  Appendix B provides a list of some elementary properties of the Meijer $G$-function and modified Bessel functions that are used in this paper.  Appendix C provides a list of inequalities for expressions involving derivatives and integrals of modified Bessel functions that are used to bound the derivatives of the solution of the $\mathrm{PN}(2,\sigma^2)$ Stein equation.   

\vspace{3mm}

\emph{Notation.} Throughout this paper, $T$ will denote the operator $Tf(x)=xf'(x)$ and $A_n$ will denote the operator $A_nf(x)=x^{-1}T^nf(x)=\frac{\mathrm{d}}{\mathrm{d}x}(T^{n-1}f(x))$.  We shall write $C^n(\mathbb{R}^d)$ for the space of $k$ times differentiable functions on $\mathbb{R}^d$.  We let $C_b^n(\mathbb{R}^d)$ denote the space of bounded functions on $\mathbb{R}^d$ with bounded $k$-th order derivatives for $k\leq n$.  We shall also write $\|f\|=\|f\|_{\infty}=\sup_{x\in\mathbb{R}}|f(x)|$.

\section{A Stein equation for products of independent central normal variables}

\subsection{The product normal Stein equation}

In this section we obtain a characterising equation for the product normal distribution which motivates the $\mathrm{PN}(n,\sigma^2)$ Stein equation (\ref{delta}). Before presenting that result, we present some useful properties of the operator $A_nf(x)=x^{-1}T^nf(x)$ and establish the existence of the zero bias distribution of order $n$ for any $W$ with zero mean and finite, non-zero variance.  

We begin by obtaining an inverse of the operator $A_n$.  This inverse operator will be used throughout this paper in establishing properties of the zero bias distribution of order $n$. 

\begin{lemma}\label{invlem}Let $V_n$ be the product of $n$ independent $U(0,1)$ random variables, and define the operator $G_n$ by $G_nf(x)=x\mathbb{E}f(xV_n)$.  Then, $G_n$ is the right-inverse of the operator $A_n$ in the sense that
\begin{equation*}A_nG_nf(x)=f(x).
\end{equation*}
Suppose now that $f\in C^n(\mathbb{R})$.  Then, for any $n\geq 1$,
\begin{equation}\label{leftinv}G_nA_nf(x)=G_1A_1f(x)=f(x)-f(0).
\end{equation}
Therefore, $G_n$ is the inverse of $A_n$ when the domain of $A_n$ is the space of all $n$ times differentiable functions $f$ on $\mathbb{R}$ with $f(0)=0$.
\end{lemma} 

\begin{proof}We begin by obtaining a useful formula for $G_nf(x)=x\mathbb{E}f(xV_n)$.  We have that
\begin{equation*}G_nf(x)=x\int_{(0,1)^n}f(xu_1\cdots u_n)\,\mathrm{d}u_1\cdots \mathrm{d}u_n,
\end{equation*}
and by a change of variables we can write
\begin{equation}\label{g1gg}G_1f(x)=\int_0^xf(t_1)\,\mathrm{d}t_1,
\end{equation}
and for $n \geq 2$,
\begin{equation}\label{g2gg}G_nf(x)=\int_0^x\!\int_0^{t_n}\dotsi\int_0^{t_2}\frac{1}{t_2t_3\cdots t_n}f(t_1)\,\mathrm{d}t_1\,\mathrm{d}t_2\cdots \mathrm{d}t_n.
\end{equation}

We now show that $G_n$ is the right-inverse of the operator $A_n$.  It is immediate from (\ref{g1gg}) that $A_1G_1f(x)=f(x)$.  For $n\geq 2$, we differentiate the right-hand side of (\ref{g2gg}) to obtain
\[TG_nf(x)=x\frac{\mathrm{d}}{\mathrm{d}x}(G_nf(x))=G_{n-1}f(x),\]
and hence $T^kG_nf(x)=G_{n-k}f(x)$.  Using this recursive formula yields
\[A_nG_nf(x)=x^{-1}T^nG_nf(x)=\frac{\mathrm{d}}{\mathrm{d}x}(T^{n-1}G_nf(x))=\frac{\mathrm{d}}{\mathrm{d}x}(G_1f(x))=f(x).\] 

Finally, we verify relation (\ref{leftinv}).  For $n\geq 2$ we have
\begin{align*}G_nA_nf(x)&=\int_0^x\!\int_0^{t_n}\dotsi\int_0^{t_2}\frac{1}{t_2\cdots  t_n}\frac{\mathrm{d}}{\mathrm{d}t_1}(T^{n-1}f(t_1))\,\mathrm{d}t_1\,\mathrm{d}t_2\cdots \mathrm{d}t_n \\
&=\int_0^x\!\int_0^{t_n}\dotsi\int_0^{t_3}\frac{1}{t_2 t_3\cdots t_n}\Big[T^{n-1}f(t_1)\Big]_{t_1=0}^{t_1=t_2}\,\mathrm{d}t_2\,\mathrm{d}t_3\cdots \mathrm{d}t_n \\
&=\int_0^x\!\int_0^{t_n}\dotsi\int_0^{t_3}\frac{1}{t_3\cdots  t_n}\cdot t_2^{-1}T^{n-1}f(t_2)\,\mathrm{d}t_2\,\mathrm{d}t_3\cdots \mathrm{d}t_n \\
&=G_{n-1}A_{n-1}f(x),
\end{align*} 
where we used that $A_nf(x)=x^{-1}T^nf(x)=\frac{\mathrm{d}}{\mathrm{d}x}(T^{n-1}f(x))$.  Iterating gives
\[G_nA_nf(x)=G_1A_1f(x)=\int_0^xf'(t_1)\,\mathrm{d}t_1=f(x)-f(0),\]
as required.
\end{proof} 

We are now in a position to prove the existence and uniqueness of the zero bias distribution of order $n$ for any $W$ with zero mean and finite, non-zero variance.  The proof is a generalisation of the proof of existence of the zero bias distribution that was given in Goldstein and Reinert \cite{goldstein}.

\begin{lemma} \label{etna} Let $W$ be a mean zero random variable with finite, non-zero variance $\sigma^2$.  Then there exists a unique random variable $W^{*(n)}$ such that for all $f\in C^n(\mathbb{R})$ for which the relevant expectations exist we have
\[\mathbb{E}Wf(W)=\sigma^2\mathbb{E}A_nf(W^{*(n)}).
\]
\end{lemma}

\begin{proof}For $f\in C_c$, the collection of continuous functions with compact support, define a linear operator $R$ by
\[Rf=\sigma^{-2}\mathbb{E}WG_nf(W),\]
where $G_n$ is defined as in Lemma \ref{invlem}.  Then $Rf$ exists, since $\mathbb{E}W^2<\infty$.  To see, moreover, that $R$ is positive, take $f\geq 0$.  Then $G_nf(x)$ is increasing, and therefore $W$ and $G_nf(W)$ are positively correlated.  Hence $\mathbb{E}WG_nf(W)\geq\mathbb{E}W\mathbb{E}fG_n(W)=0$, and $R$ is positive.  Using the Riesz representation theorem (see, for example, Folland \cite{folland}) we have $Rf=\int f\,\mathrm{d}\nu$, for some unique Radon measure $\nu$, which is a probability measure as $R1=1$.  

We now take $f(x)=A_ng(x),$ where $g\in C^n(\mathbb{R})$, with derivatives up to $n$-th order being continuous with compact support.  Then, from (\ref{leftinv}), we have
\[\mathbb{E}WG_nA_ng(W)=\mathbb{E}W(g(W)-g(0))=\mathbb{E}Wg(W),\]
which completes the proof.
\end{proof}

The following proposition motivates the product normal Stein equation (\ref{delta}).

\begin{proposition}\label{prodsteinlemma}Suppose $Z\sim\mathrm{PN}(n,\sigma^2)$.  Let $f\in C^n(\mathbb{R})$ be such that $\mathbb{E}|Zf(Z)|<\infty$ and $\mathbb{E}|Z^{k-1}f^{(k)}(Z)|<\infty$, $k=1,\ldots,n$.  Then
\begin{equation}\label{cracker} \mathbb{E}[\sigma^2A_nf(Z)-Zf(Z)]=0.
\end{equation}
\end{proposition}

\begin{proof}Define $W_n=\prod_{i=1}^nX_i$ where $X_i\sim N(0,\sigma_{x_i}^2)$ and the $X_i$ are independent.  We also let $\sigma_n=\sigma_{x_1}\sigma_{x_2}\cdots\sigma_{x_n}$ and observe that $(T^nf)(ax)=T^nf_{a}(x)$ where $f_{a}(x)=f(ax)$.  

We prove the result by induction on $n$.  For $n=1$, the result follows immediately from the characterisation (\ref{stein lemma}) for the normal distribution.  By induction assume that $\mathbb{E}W_ng(W_n)=\sigma_n^2\mathbb{E}A_ng(W_n)=\sigma_{n}^2\mathbb{E} W_n^{-1}T^ng(W_n)$ for all $g\in C^n(\mathbb{R})$ for some $n\geq 1$.  Then 
\begin{align*}
\mathbb{E}W_{n+1}f(W_{n+1}) &=\mathbb{E}[X_{n+1}\mathbb{E}[W_n f_{X_{n+1}}(W_n)\mid X_{n+1}]]\\
&= \mathbb{E}[X_{n+1}\mathbb{E}[\sigma_{n}^2W_{n}^{-1}(T^nf_{X_{n+1}})(W_n)\mid X_{n+1}]]\\
&= \sigma_{n}^2\mathbb{E}[X_{n+1}W_{n}^{-1}(T^nf)(W_nX_{n+1})]\\
&= \sigma_{n}^2\mathbb{E}[W_{n}^{-1}\mathbb{E}[X_{n+1}\cdot (T^nf_{W_n})(X_{n+1})\mid W_n]]\\
&= \sigma_{n+1}^2\mathbb{E}[W_{n}^{-1}\mathbb{E}[X_{n+1}^{-1}(T^{n+1}f_{W_n})(X_{n+1})\mid W_n]]\\
&= \sigma_{n+1}^2\mathbb{E}[W_{n+1}^{-1}T^{n+1}f(W_{n+1})],
\end{align*} 
as required.
\end{proof}

\begin{remark}\label{nice eqn}\emph{We could have obtained a first order Stein operator for the $\mathrm{PN}(n,\sigma^2)$ distributions using the density approach of Stein et al$.$ \cite{stein3} (see also Ley et al$.$ \cite{ley} for an extension of the scope of the density method).  However, this approach would lead to complicated operators involving Meijer $G$-functions, which, in contrast to our Stein equation, may not be amenable to the use of couplings.} 
\end{remark}

\subsection{Applications of Proposition \ref{prodsteinlemma}}

The main application of Proposition \ref{prodsteinlemma} that is considered in this paper involves the use of
the resulting $\mathrm{PN}(2,\sigma^2)$ Stein equation in the proofs of the limit theorems of Section 4. It is,
however, possible to obtain other interesting results using Proposition \ref{prodsteinlemma}.  As an example, we demonstrate how the characterising equation (\ref{cracker}) of the product normal distributions can be used to derive a formula for the characteristic function of the $\mathrm{PN}(n,\sigma^2)$ distribution.

Let $W_n\sim\mathrm{PN}(n,\sigma^2)$.  We begin by noting that the moment generating function of $W_n$ is only defined for $n\leq 2$ (see (\ref{pnxinfty})).  We therefore consider the characteristic function of $W_n$.  On taking $f(x)=\mathrm{e}^{itx}$ in the characterising equation (\ref{cracker}) and setting $\phi_n(t)=\mathbb{E}[\mathrm{e}^{itW_n}]$, we deduce that $\phi_n(t)$ satisfies the differential equation
\begin{equation}\label{charmei}\sigma^2t\bigg(t\frac{\mathrm{d}}{\mathrm{d}t}+1\bigg)^{n-1}\phi_n(t)+\phi_n'(t)=0.
\end{equation}
Solving this differential equation subject to the conditions that $\phi_n(t)$ is the characteristic function of $W_n$ would yield a formula for the characteristic function of $W_n$.  For simplicity, we consider the cases $n=2,3$ (the case $n=1$ gives the well-known formula $\phi_1(t)=\mathrm{e}^{-\frac{1}{2}\sigma^2t^2}$).  

For $n=2$, the differential equation (\ref{charmei}) reduces to
\begin{equation*}(1+\sigma^2t^2)\phi_2'(t)+\sigma^2t\phi_2(t)=0.
\end{equation*}  
Solving this equation subject to the condition $\phi_2(0)=1$ gives
\[\phi_2(t)=\frac{1}{\sqrt{1+\sigma^2t^2}}.\]
This formula was also obtained in Example 11.22 of Stuart and Ord \cite{stuart}.

For $n=3$, the characteristic function $\phi_3(t)$ satisfies
\begin{equation*}\sigma^2t^3\phi_3''(t)+(3\sigma^2t^2+1)\phi_3'(t)+\sigma^2t\phi_3(t)=0.
\end{equation*}
It is straightforward to verify the general solution of this differential equation is given by
\begin{equation*}\phi_3(t)=A|t|^{-1}\exp\bigg(\frac{1}{4\sigma^2t^2}\bigg)I_0\bigg(\frac{1}{4\sigma^2t^2}\bigg)+B|t|^{-1}\exp\bigg(\frac{1}{4\sigma^2t^2}\bigg)K_0\bigg(\frac{1}{4\sigma^2t^2}\bigg),
\end{equation*}
where $I_0(x)$ and $K_0(x)$ are modified Bessel functions (see Appendix B.2.1 for definitions), which satisfy the modified Bessel differential equation (\ref{realfeel}).  From the asymptotic formula $I_0(x)\sim\frac{\mathrm{e}^x}{\sqrt{2\pi x}}$ as $x\rightarrow\infty$ (see (\ref{roots})), it follows that the solution is unbounded as $t\rightarrow0$ unless $A=0$.  Since the characteristic function $\phi_3(t)$ must be bounded for all $t\in\mathbb{R}$, we take $A=0$.  We also require that $\phi_3(0)=1$, and so on applying the asymptotic formula $K_0(x)\sim\sqrt{\frac{\pi}{2x}}\mathrm{e}^{-x}$ as $x\rightarrow\infty$ (see (\ref{Ktendinfinity})), we take $B=(\sigma\sqrt{2\pi})^{-1}$.  Hence,
\[\phi_3(t)=\frac{1}{\sqrt{2\pi\sigma^2t^2}}\exp\bigg(\frac{1}{4\sigma^2t^2}\bigg)K_0\bigg(\frac{1}{4\sigma^2t^2}\bigg).\]

In principle, this approach could be used to obtain a formula for $\phi_n(t)$ for any $n\geq1$.  Although, for $n\geq 4$, the general solution is given in terms of Meijer $G$-functions and it would become increasing difficult to use the condition that $\phi_n(t)$ is the characteristic function of the $\mathrm{PN}(n,\sigma^2)$ distribution to determine the values of the constants.  We can, however, obtain a formula for $\phi_n(t)$ by using an integral formula involving the Meijer $G$-function, given by formula 18 of Section 5.6 of Luke \cite{luke} (see \ref{meijergintegration}).  As far as the author is aware, the formula in the following proposition is new. 

\begin{proposition}\label{propcharmeijer}The characteristic function of the $\mathrm{PN}(n,\sigma^2)$ distribution is given by
\begin{equation*}\phi_n(t)=\frac{1}{\pi^{(n-1)/2}}G_{1,n-1}^{n-1,1}\bigg(\frac{1}{2^{n-2}\sigma^2t^2}\;\bigg|\;\begin{matrix} 1 \\
\frac{1}{2},\ldots,\frac{1}{2} \end{matrix} \bigg).
\end{equation*} 
In the cases $n=2,3$, this formula simplifies to
\begin{equation*}\phi_2(t)=\frac{1}{\sqrt{1+\sigma^2t^2}} \qquad \text{and} \qquad \phi_3(t)=\frac{1}{\sqrt{2\pi\sigma^2t^2}}\exp\bigg(\frac{1}{4\sigma^2t^2}\bigg)K_0\bigg(\frac{1}{4\sigma^2t^2}\bigg).
\end{equation*}
\end{proposition}

\begin{proof}Let $W_n\sim\mathrm{PN}(n,\sigma^2)$.  Since the $\mathrm{PN}(n,\sigma^2)$ distribution is symmetric about the origin, it follows that the characteristic function $\phi_n(t)$ of $W_n$ is given by
\begin{equation*}\phi_n(t)=\mathbb{E}[\mathrm{e}^{itW_n}]=\mathbb{E}[\cos(tW_n)]=2\int_0^{\infty}\frac{1}{(2\pi)^{n/2}\sigma}\cos(tx)G_{0,n}^{n,0}\bigg(\frac{x^2}{2^n\sigma^2} \; \bigg| \;0,\ldots,0\bigg)\,\mathrm{d}x.
\end{equation*}
Evaluating the integral using (\ref{meijergintegration}) and then simplifying using (\ref{meijergidentity}) and (\ref{lukeformula}) gives
\begin{align*}\phi_n(t)&=\frac{1}{2^{n/2-1}\pi^{(n-1)/2}\sigma|t|}G_{2,n}^{n,1}\bigg(\frac{1}{2^{n-2}\sigma^2t^2} \; \bigg| {\frac{1}{2},0 \atop 0,\ldots,0}\bigg)\\
&=\frac{1}{\pi^{(n-1)/2}}G_{2,n}^{n,1}\bigg(\frac{1}{2^{n-2}\sigma^2t^2} \; \bigg| {1,\frac{1}{2} \atop \frac{1}{2},\ldots,\frac{1}{2}}\bigg)\\
&=\frac{1}{\pi^{(n-1)/2}}G_{1,n-1}^{n-1,1}\bigg(\frac{1}{2^{n-2}\sigma^2t^2} \; \bigg| {1 \atop \frac{1}{2},\ldots,\frac{1}{2}}\bigg),
\end{align*}
as required.
\end{proof}

\subsection{Bounds for the solution}

Here we solve the product normal Stein equation.  For the case $n=2$, the solution can represented in terms of integrals of modified Bessel functions and takes a relatively simple form.  As a result, we are able to obtain uniform bounds for the solution and its lower order derivatives.  These bounds are used in the approximation theorems of Section 4.  However, for $n\geq3$, the solution takes a less tractable form and we have been unable to bound the solution or any of its derivatives.

\subsubsection{$n=2$}

The $\mathrm{PN}(2,\sigma^2)$ Stein equation (\ref{deltappp}) is a second order linear differential equation and the homogeneous equation has $I_0(x/\sigma)$ and $K_0(x/\sigma)$ as a pair of linearly independent solutions (see (\ref{realfeel})).  Therefore, one can obtain a solution to the $\mathrm{PN}(2,\sigma^2)$ Stein equation through a straightforward application of the method of variation of parameters (see Collins \cite{collins} for an account of the method).  This was done by Gaunt \cite{gaunt vg}, Lemma 3.3 for the more general variance-gamma Stein equation (recall that the $\mathrm{PN}(2,\sigma^2)$ Stein equation is a special case of the variance-gamma Stein equation), and the following lemma is a special case of that result.  We present the solution for the case $\sigma=1$; we can recover results for the general case by using a simple change of variables.    

\begin{lemma} \label{forty}
Suppose $h:\mathbb{R} \rightarrow \mathbb{R}$ is bounded.  Then the unique bounded solution $f:\mathbb{R} \rightarrow \mathbb{R}$ to the $\mathrm{PN}(2,1)$ Stein equation (\ref{deltappp}) is given by
\begin{align} \label{ink} f(x) &=-K_0(|x|) \int_0^x I_0(y) [h(y)- \mathrm{PN}_2^1h] \,\mathrm{d}y - I_0(x) \int_x^{\infty}  K_0(|y|)[h(y)- \mathrm{PN}_2^1h]\,\mathrm{d}y,
\end{align}
where $I_0(x)$ and $K_0(x)$ are modified Bessel functions; see Appendix B.2.1 for definitions.
\end{lemma}

\begin{remark} \label{mark} \emph{The equality 
\begin{equation*} \int_{-\infty}^{x}K_0(|y|)[h(y)-\mathrm{\mathrm{PN}}_2^1h]\,\mathrm{d}y=-\int_{x}^{\infty}K_0(|y|)[h(y)-\mathrm{\mathrm{PN}}_2^1h]\,\mathrm{d}y
\end{equation*}
is very useful when it comes to obtaining bounds for the derivatives of the solution to the Stein equation.  The equality ensures that we can restrict out attention to bounding the derivatives in the region $x\geq 0$.}
\end{remark}

By direct calculations it is possible to obtain bounds for the derivatives of the solution of the $\mathrm{\mathrm{PN}}(2,1)$ Stein equation.  Bounds for the case of general $\sigma$ then follow by a simple change of variables.  In Theorem \ref{compoz7} (see Appendix A.2$.$ for the proof) we present uniform bounds for the solution of the $\mathrm{\mathrm{PN}}(2,\sigma^2)$ Stein equation and its first four derivatives in terms of the supremum norms of the derivatives of the test function $h$.  It is worth noting that bounds for the solution of the variance-gamma Stein equation and its first four derivatives  were derived by Gaunt \cite{gaunt vg}, yielding as a special case estimates for the solution of the $\mathrm{PN}(2,\sigma^2)$ Stein equation.  Also, through a new iterative technique, D\"{o}bler et al$.$ \cite{dgv} extended this work to obtain bounds on derivatives of arbitrary order of the solution of variance-gamma Stein equation.  However, the constants in Theorem \ref{compoz7} improve on these for $n\leq4$, and the estimates for $\|xf(x)\|$, $\|xf'(x)\|$ and $\|xf''(x)\|$ are new.

By exploiting properties of the zero bias transformation of order $n$, we only require bounds for derivatives up to second order for the limit theorems of Section 4; in fact for Corollary \ref{halifax stupid} we only need estimates for the solution and its first derivative.  However, for other coupling choices, such as local couplings, we would require bounds on the third derivative to achieve a $O(m^{-1/2})$ bound (here $m$ is the `sample size'); and for $O(m^{-1})$ bounds we may require bounds for the fourth derivative (see, for example, Goldstein and Reinert \cite{goldstein}, Corollary 3.1).   

\begin{theorem}\label{compoz7}Suppose that $h\in C_b^3(\mathbb{R})$.  Then the solution $f$ of the $\mathrm{\mathrm{PN}}(2,\sigma^2)$ Stein equation (\ref{deltappp}) and its first four derivatives are bounded as follows 
\begin{eqnarray*}\|f\|&=& \frac{3}{\sigma}\|h-\mathrm{\mathrm{PN}}_2^{\sigma^2}h\|, \\
\|f'\| &=& \frac{3}{2\sigma^2}\|h-\mathrm{\mathrm{PN}}_2^{\sigma^2}h\|, \\
\|f''\| &=& \frac{2\|h'\|}{\sigma^2}+\frac{5}{\sigma^3}\|h-\mathrm{\mathrm{PN}}_2^{\sigma^2}h\|, \\
\|f^{(3)}\| &=& \frac{4\|h''\|}{\sigma^2}+\frac{5\|h'\|}{\sigma^3}+\frac{4.89}{\sigma^4}\|h-\mathrm{\mathrm{PN}}_2^{\sigma^2}h\|, \\
\|f^{(4)}\| &=& \frac{8\|h^{(3)}\|}{\sigma^2}+\frac{9\|h''\|}{\sigma^3}+\frac{6.81\|h'\|}{\sigma^4}+\frac{15.75}{\sigma^5}\|h-\mathrm{\mathrm{PN}}_2^{\sigma^2}h\|,
\end{eqnarray*}
and we also have
\begin{eqnarray*}\|xf(x)\|&\leq& \frac{2}{\sigma}\|h-\mathrm{\mathrm{PN}}_2^{\sigma^2}h\|, \\
\|xf'(x)\|&\leq&\frac{3}{2\sigma^2}\|h-\mathrm{\mathrm{PN}}_2^{\sigma^2}h\|, \\
\|xf''(x)\|&\leq&\frac{9}{2\sigma^3}\|h-\mathrm{\mathrm{PN}}_2^{\sigma^2}h\|.
\end{eqnarray*}
\end{theorem}

We now prove a lemma, which gives a bound on the quantities $\|(A_nf)^{(k)}\|$ that appear in the bounds of Theorem \ref{jazzz}.  In the lemma, we suppose that certain derivatives of the solution of the $\mathrm{\mathrm{PN}}(n,\sigma^2)$ Stein equation (\ref{delta}) exist and are bounded.  However, this might not be the case when $n\geq 3$, as we have not been able to obtain an analogue of Theorem \ref{compoz7} for $n\geq3$; this is discussed further in Section 2.3.2.

\begin{lemma}\label{arflem}Let $f$ be the solution of the $\mathrm{\mathrm{PN}}(n,\sigma^2)$ Stein equation (\ref{delta}).  Suppose that $f$ and the test function $h$ are $k$ times differentiable.  Then
\begin{equation}\label{arf1}\sigma^2\|(A_nf)^{(k)}\|\leq \|h^{(k)}\|+\|xf^{(k)}(x)\|+k\|f^{(k-1)}\|,
 \end{equation}
provided that the supremum norms on the right-hand side of (\ref{arf1}) are bounded.  In particular,
\begin{eqnarray*}\sigma^2\|(A_2f)'\|&\leq&\|h'\|+\bigg(\frac{3}{\sigma}+\frac{3}{2\sigma^2}\bigg)\|h-\mathrm{\mathrm{PN}}_2^{\sigma^2}h\|,\\
\sigma^2\|(A_2f)''\|&\leq&\|h''\|+\bigg(\frac{3}{\sigma^2}+\frac{9}{2\sigma^3}\bigg)\|h-\mathrm{\mathrm{PN}}_2^{\sigma^2}h\|.
\end{eqnarray*}
\end{lemma}

\begin{proof}The solution $f$ satisfies the equation $\sigma^2A_nf(x)=h(x)-xf(x)$.  By a simple induction, $\sigma^2(A_nf)^{(k)}(x)=h^{(k)}(x)+xf^{(k)}(x)+kf^{(k-1)}(x)$.  Applying the triangle inequality now yields (\ref{arf1}).  To obtain the inequalities for $\|(A_2f)'\|$ and $\|(A_2f)''\|$, we use the bounds for $f$ and its derivatives that were given in Theorem \ref{compoz7} to bound the supremum norms on the right-hand side of (\ref{arf1}) for the case $n=2$ and $k=1,2$.
\end{proof}

\subsubsection{$n\geq3$}

We now consider the more challenging $n\geq3$ case.  Again, for simplicity, we shall take $\sigma=1$; results for the general case easily follow from a change of variables.  As in the case $n=2$, we can obtain a fundamental system for the homogeneous equation, and can thus use variation of parameters to write down a solution to the Stein equation. 

\begin{lemma}\label{ngeq3lem}Suppose $h:\mathbb{R}\rightarrow\mathbb{R}$ is bounded, and set $\tilde{h}(x)=h(x)-\mathrm{PN}_n^1 h$.  Let $y_k(x)=G_{0,n}^{k,0}((-1)^kx^2/2^n\;|\;0,\ldots,0)$.  Then the function
\begin{equation}\label{pn1soln}f(x)=(-1)^{n}y_1(x)\int_x^{\infty}\frac{W_1(t)\tilde{h}(t)}{t^{n-1}W(t)}\,\mathrm{d}t+\sum_{k=2}^n (-1)^{n+k} y_k(x)\int_0^x\frac{W_k(t)\tilde{h}(t)}{t^{n-1}W(t)}\,\mathrm{d}t
\end{equation}
solves the $\mathrm{PN}(n,1)$ Stein equation, where $W$ is the determinant of the matrix
\begin{equation}\label{mnji}\mathbf{M}=\left( \begin{array}{cccc}
y_1 & y_2 & \cdots & y_n \\
y_1' & y'_2 & \cdots & y_n' \\
\vdots & \vdots & \ddots & \vdots \\
y_1^{(n-1)} & y_2^{(n-1)} & \cdots & y_n^{(n-1)} \end{array} \right)
\end{equation}
and $W_k$ is the determinant of the submatrix of $\mathbf{M}$ obtained by deleting the last row and $k$-th column.

Moreover, there is at most one bounded solution to the $\mathrm{PN}(n,1)$ Stein equation.
\end{lemma}

The representation of the solution (\ref{pn1soln}), which is given in terms of integrals involving expressions of Meijer $G$-functions, does not lend itself easily to the bounding of the solution and its derivatives.  This is in contrast to the relatively simple solution (\ref{ink}) that corresponds to the $n=2$ case.  For this reason, we have not been able to obtain an extension of Theorem \ref{compoz7} to $n\geq3$, and this is left as an interesting open problem.

There are various ways one could approach this problem.  One would be to use a different method to solve the Stein equation, which may lead to a different (and hopefully simpler) representation of the solution.  Alternatively, a first step would be to obtain a simplification of the current representation (\ref{pn1soln}), perhaps by simplifying the determinants  $W_1,\ldots,W_n$ and $W$.  For $n=2$, we have $y_1(x)=G_{0,2}^{1,0}(-x^2/4\;|\;0,0)=I_0(x)$ and $y_2(x)=G_{0,2}^{2,0}(x^2/4\;|\;0,0)=2K_0(|x|)$ (see (\ref{iomei}) and (\ref{komei})).  Therefore, for $x>0$,
\begin{align*}W(x)&=W(y_1(x),y_2(x))=y_1(x)y_2'(x)-y_1'(x)y_2(x) \\
&=-2(I_0(x)K_1(x)+I_1(x)K_0(x))=-\frac{2}{x}, \quad \text{(by (\ref{wront}))}
\end{align*}
with a similar formula for $x<0$, and one could attempt to obtain a similar simplification for $n\geq3$.  From here, to bound the solution and its derivatives, one could proceed, through rather technical calculations, to bound appropriate expressions involving integrals of Meijer $G$-functions.  These calculations would be similar, but more technical, than those used to prove Theorem \ref{compoz7} (see Appendix A, and Gaunt \cite{gaunt} for the calculations used to obtain the bounds of Appendix C).

If instead of seeking explicit bounds, as was the case in Theorem \ref{compoz7}, one was content to just prove that the solution and its lower order derivatives were bounded, for sufficiently differentiable test functions $h$, then one would only need to investigate the behaviour of the solution (\ref{pn1soln}) in the limits $x\rightarrow0$ and $|x|\rightarrow\infty$, as the solution and its derivatives are bounded for all other $x\in\mathbb{R}$.  Asymptotic formulas for the $G$-function are available (see, for example, Luke \cite{luke}, section 5.10); however, such an analysis may be rather involved, due to occurrence of the Wronskian determinants in the solution (\ref{pn1soln}).  Although, the complexity of this analysis would be greatly reduced if a simplification was found for the determinants.   Finally, due to Lemma \ref{arflem}, if one was using the zero bias transformation of order $n$ to obtain product normal approximation results, then it would suffice to bound just $\|f\|$ and $\|xf'(x)\|$.

\vspace{3mm}

\noindent\emph{Proof of Lemma \ref{ngeq3lem}.}  We begin by stating a general result that is given in Kreyszig \cite{kreyszig}, p$.$ 140.  Consider the inhomogeneous differential equation 
\begin{equation*}q^{(n)}(x)+a_{n-1}(x)q^{(n-1)}(x)+\cdots+a_1(x)q'(x)+a_0(x)q(x)=g(x),
\end{equation*}
where the $a_k(x)$ are $(n-1)$-times differentiable.  Suppose that $q_1(x),\ldots,q_n(x)$ form a fundamental solution to the homogeneous equation.  Then the general solution is given by
\begin{equation}\label{vdfkjbv}q(x)=\sum_{k=1}^n (-1)^{n+k} q_k(x)\int_{a_k}^x\frac{W_k(t)g(t)}{W(t)}\,\mathrm{d}t,
\end{equation}
where the $a_k$ are arbitrary constants, and $W$ and $W_k$ are determinants of a matrix of the form (\ref{mnji}).

We now apply this result to the $\mathrm{PN}(n,1)$ Stein equation.  The homogeneous equation is
\begin{equation}\label{hom1}x^{-1}T^nf(x)-xf(x)=0
\end{equation}
and letting $f(x)=g((-1)^k x^2/2^n)$ leads to the differential equation
\begin{equation}\label{hom2}T^ng(x)-(-1)^k xg(x)=0.
\end{equation}
This is a special case of the Meijer $G$-function differential equation (\ref{meijergdiff}), and the functions $G_{0,n}^{k,0}(x\,|\,0,\ldots,0)$, $k=1,\ldots,n$, form a fundamental set of solutions to (\ref{hom2}).  Consequently, the functions $y_k(x)=G_{0,n}^{k,0}((-1)^k x^2/2^n\,|\,0,\ldots,0)$, $k=1,\ldots,n$, form a fundamental set of solutions to (\ref{hom1}).  The differential equation (\ref{hom1}) is a $n$-th order linear differential equation with $x^{n-1}$ as coefficient of $f^{(n)}$, and so from (\ref{vdfkjbv}) we have that (\ref{pn1soln}) solves the $\mathrm{PN}(n,1)$ Stein equation.  

The values of the arbitrary constants were chosen to allow the solution to be bounded.  In fact, in order to have a bounded solution to the $\mathrm{PN}(n,1)$ Stein equation, one must take $a_2,\ldots,a_n=0$ and either $a_1=\infty$ or $a_1=-\infty$.  This is because, due to asymptotic properties of Meijer $G$-functions (see Luke \cite{luke}, section 5.10), we have that $y_1(x)\rightarrow\infty$ as $|x|\rightarrow\infty$ and, for $k=2,\ldots,n$, $y_k(x)\rightarrow\infty$ as $x\rightarrow0$ (asymptotic formulas are only given for $|x|\rightarrow\infty$ in \cite{luke}, although asymptotic results for the limit $x\rightarrow0$ can be obtained via identity (\ref{meijerg-1})).

Finally, we prove that there is at most one bounded solution to the Stein equation.  Suppose $u$ and $v$ are bounded solutions to the Stein equation.  Define $w=u-v$.  Then $w$ is also bounded and is a solution to (\ref{hom1}), which has general solution $w(x) = \sum_{k=1}^nA_ky_k(x)$,
where the $A_k$ are arbitrary constants.  Recall that $y_1(x)\rightarrow\infty$ as $|x|\rightarrow\infty$ and, for $k=2,\ldots,n$, $y_k(x)\rightarrow\infty$ as $x\rightarrow0$.  Therefore, for $w$ to be bounded, we must take $A_1=\ldots=A_n=0$, and thus $w=0$ and so there can be at most one bounded solution.   \hfill $\square$ 

\section{The zero bias transformation of order $n$}
The zero bias transformation of order $n$, given in Definition \ref{nthzero}, has many useful properties, which we collect in Propositions \ref{doublesquare} and \ref{john hates}, below.  We begin by presenting a relationship between the $W$-zero bias distribution of order $n$ and the $W$-\emph{square bias distribution} (see Chen et al$.$ \cite{chen}, pp$.$ 34--35, for properties of this distributional transformation).  For any random variable $W$ with finite second moment, we say that $W^{\square}$ has the $W$-square bias distribution if for all $f$ such that $\mathbb{E}W^2f(W)$ exists,
\begin{equation*} \mathbb{E}W^2f(W)=\mathbb{E}W^2\mathbb{E}f(W^{\square}).
\end{equation*}
When $W$ is non-negative, $W^{\square}$ is given by the size bias distribution of $W^s$, the size bias distribution of $W$ (see Pekoz et al$.$ \cite{pekoz}).  Although, in this section, we will be considering $W$ with mean zero, and so this attractive result is not available to us.

Before presenting our relationship, we write down a construction of the $W$-square bias distribution, for the case that $W$ is decomposed into a product of independent random variables.  This construction will be used in the proof of part (iv) of Proposition \ref{john hates}.
\begin{proposition}\label{zerobox}Suppose $W=\prod_{k=1}^nW_k$, where the $W_k$ are independent random variables and let $W_1^{\square},\ldots,W_n^{\square}$ be independent random variables with $W_k^{\square}$ having the $W_k$-square biased distribution, then
\begin{equation*}\label{squareconst}W^{\square}=\prod_{k=1}^nW_k^{\square}
\end{equation*}
has the $W$-square biased distribution.
\end{proposition}

\begin{proof}We prove that the result holds for the case of two products; the extension to a general number of products follows by a straightforward induction.  Using independence to obtain the final equality, we have 
\begin{align*}\mathbb{E}f(W_1^{\square}W_2^{\square})&=\mathbb{E}[\mathbb{E}[f(W_1^{\square}W_2^{\square})\,|\,W_2^{\square}]] \\
&=\frac{1}{\mathbb{E}W_1^{2}}\mathbb{E}[W_1^{2}\mathbb{E}[f(W_1W_2^{\square})\,|\,W_2^{\square}]] \\
&=\frac{1}{\mathbb{E}W_1^{2}}\mathbb{E}W_1^{2}f(W_1W_2^{\square}) \\
&=\frac{1}{\mathbb{E}W_1^{2}\mathbb{E}W_2^{2}}\mathbb{E}W_1^{2}W_2^2f(W_1W_2) \\
&=\frac{1}{\mathbb{E}W_1^{2}W_2^{2}}\mathbb{E}W_1^{2}W_2^2f(W_1W_2),
\end{align*}
as required.
\end{proof}

We now state our relationship, which is a natural generalisation of the relation between the zero bias distribution and the square bias distribution that is given in Proposition 2.3 of Chen et al$.$ \cite{chen}.  The proof of the relationship utilises the inverse operator $G_n$ of $A_n$, and differs from the approach used in \cite{chen}. 

\begin{proposition}\label{doublesquare}Let $W$ be a random variable with zero mean and finite, non-zero variance $\sigma^2$, and let $W^{\square}$ have the $W$-square bias distribution.  Let $U_1,\ldots,U_n$ be independent $U(0,1)$ random variables, which are independent of $W^{\square}$.  Define $V_n=\prod_{k=1}^nU_k$.  Then, the random variable
\begin{equation*}W^{*(n)}\stackrel{\mathcal{D}}{=}V_nW^{\square}
\end{equation*}
has the $W$-zero bias distribution of order $n$.
\end{proposition}

\begin{proof}Let $f\in C_c$, the collection of continuous functions with compact support.  In Lemma \ref{invlem} we defined the operator $G_ng(x)=x\mathbb{E}g(xV_n)$ and showed that $A_nG_ng(x)=g(x)$ for any $g$.  We therefore have
\[\sigma^2\mathbb{E}f(W^{*(n)})=\sigma^2\mathbb{E}A_nG_nf(W^{*(n)})=\mathbb{E}WG_nf(W)=\mathbb{E}W^2f(V_nW) =\sigma^2\mathbb{E}f(V_nW^{\square}).\]
Since the expectation of $f(W^{*(n)})$ and $f(V_nW^{\square})$ are equal for all $f\in C_c$, the random variables $W^{*(n)}$ and $V_nW^{\square}$ must be equal in distribution.
\end{proof}

In the following proposition we present some properties of the zero bias transformation of order $n$.  These properties generalise some of the important properties of the zero bias transformation (see Goldstein and Reinert \cite{goldstein}, Lemma 2.1 and Chen et al$.$ \cite{chen}, Proposition 2.1).  As was the case for the proof of Proposition \ref{doublesquare}, the proofs of parts (i) and (ii) of the proposition utilise the inverse operator $G_n$, and so differ from the approach given in \cite{chen} and\cite{goldstein}. 

\begin{proposition} \label{john hates} Let $W$ be a mean zero variable with finite, non-zero variance $\sigma^2$, and let $W^{*(n)}$ have the $W$-zero biased distribution of order $n$ in accordance with Definition \ref{nthzero}.

(i) The distribution function of $W^{*(n)}$ is given by
\begin{equation} \label{asdf} F_{W^{*(n)}}(w)=\begin{cases}  \displaystyle \frac{1}{(n-1)!\sigma^2}\mathbb{E}\bigg[W^2\gamma\bigg(n,\log\bigg(\frac{W}{w}\bigg)\bigg)\mathbf{1}(W\leq w)\bigg], & \quad w<0, \\[10pt]

 \displaystyle 1-\frac{1}{(n-1)!\sigma^2}\mathbb{E}\bigg[W^2\gamma\bigg(n,\log\bigg(\frac{W}{w}\bigg)\bigg)\mathbf{1}(W\geq w)\bigg], & \quad w\geq 0,
\end{cases}
\end{equation}
where $\gamma(a,x)$ denotes the lower incomplete gamma function $\gamma(a,x)=\int_0^xt^{a-1}\mathrm{e}^{-t}\,\mathrm{d}t$.

(ii) The distribution of $W^{*(n)}$ is unimodal about zero with density
\begin{equation} \label{zxcv} f_{W^{*(n)}}(w)=\begin{cases}  \displaystyle -\frac{1}{(n-1)!\sigma^2}\mathbb{E}\bigg[W\bigg(\log \bigg(\frac{W}{w}\bigg)\bigg)^{n-1}\mathbf{1}(W\leq w)\bigg], & \quad w<0, \\[10pt]

\displaystyle\frac{1}{(n-1)!\sigma^2}\mathbb{E}\bigg[W\bigg(\log \bigg(\frac{W}{w}\bigg)\bigg)^{n-1}\mathbf{1}(W\geq w)\bigg], & \quad w\geq 0.
\end{cases}
\end{equation}
It follows that the support of $W^{*(n)}$ is the closed convex hull of the support of $W$ and that $W^{*(n)}$ is bounded whenever $W$ is bounded.

Suppose now that $W$ is symmetric, then the density of $W^{*(n)}$ is given by
 \begin{align}\label{zeropdf1}f_{W^{*(n)}}(w)&=  \frac{1}{(n-1)!\sigma^2}\mathbb{E}\bigg[W\bigg(\log \bigg|\frac{W}{w}\bigg|\bigg)^{n-1}\mathbf{1}(W\geq w)\bigg] \\ 
\label{zeropdf2}&= -\frac{1}{(n-1)!\sigma^2}\mathbb{E}\bigg[W\bigg(\log \bigg|\frac{W}{w}\bigg|\bigg)^{n-1}\mathbf{1}(W\leq w)\bigg].
\end{align} 
In particular, the zero bias transformation of order $n$ preserves symmetry.

(iii) For $p\geq 0$ we have
\begin{equation*}\mathbb{E}(W^{*(n)})^p=\frac{\mathbb{E}W^{p+2}}{\sigma^2(p+1)^n} \qquad \mbox{and} \qquad
\mathbb{E}|W^{*(n)}|^p = \frac{\mathbb{E}|W|^{p+2}}{\sigma^2(p+1)^n}.
\end{equation*} 

(iv) Suppose $W=\prod_{k=1}^nW_k$, where the $W_k$ are independent random variables with zero mean and finite, non-zero variance and $W_1^*,\ldots,W_n^*$ are independent random variables with $W_k^*$ having the $W_k$-zero biased distribution.  Then
\begin{equation}\label{zeroconst}W^{*(n)}\stackrel{\mathcal{D}}{=}\prod_{k=1}^nW_k^*
\end{equation}
has the $W$-zero biased distribution of order $n$.

(v) For $c\in\mathbb{R}$, $cW^{*(n)}$ has the $cW$-zero biased distribution of order $n$.
\end{proposition}

\begin{proof}(i) In the proof of Proposition \ref{doublesquare}, we showed that $\sigma^2\mathbb{E}f(W^{*(n)})=\mathbb{E}W^2f(V_nW)$ for all bounded functions $f$, where $V_n$ is the product of $n$ many independent $U(0,1)$ random variables.  By taking $f(x)=\mathbf{1}(x\leq w)$ we have
\[F_{W^{*(n)}}(w)=\frac{1}{\sigma^2}\mathbb{E}[W^2\mathbf{1}(V_nW\leq w)]=1-\frac{1}{\sigma^2}\mathbb{E}[W^2\mathbf{1}(V_nW\geq w)],\]
as $\mathbb{E}W^2=\sigma^2$.  Formula (\ref{asdf}) now follows on noting that $-\log(V_n)$ follows the $\mathrm{Gamma}(n,1)$ distribution.

(ii) We verify that (\ref{zxcv}) holds for $w<0$; the proof for $w\geq 0$ is similar.  Suppose $a<0$, then the function $\gamma(n,\log(a/w))$ is differentiable on $(a,0)$, with derivative
\begin{equation}\label{qwert}\frac{\mathrm{d}}{\mathrm{d}w}\bigg[\gamma\bigg(n,\log\bigg(\frac{a}{w}\bigg)\bigg)\bigg]=-\frac{1}{a}\bigg(\log\bigg(\frac{a}{w}\bigg)\bigg)^{n-1}.
\end{equation} 
Using (\ref{qwert}) and dominated convergence, we have, for $w<0$,
\begin{align*}f_{W^{*(n)}}(w)&=\frac{1}{(n-1)!\sigma^2}\mathbb{E}\bigg[W^2\frac{\mathrm{d}}{\mathrm{d}w}\bigg[\gamma\bigg(n,\log\bigg(\frac{W}{w}\bigg)\bigg)\mathbf{1}(W\leq w)\bigg]\bigg] \\
&=-\frac{1}{(n-1)!\sigma^2}\mathbb{E}\bigg[W\bigg(\log \bigg(\frac{W}{w}\bigg)\bigg)^{n-1}\mathbf{1}(W\leq w)\bigg].
\end{align*}

It follows from (\ref{zxcv}) that $f_{W^{*(n)}}$ is increasing for $w<0$ and decreasing for $w>0$, and is thus unimodal about zero.

We now consider the case that $W$ is symmetric, and verify formulas (\ref{zeropdf1}) and (\ref{zeropdf2}) for $f_{W^{*(n)}}(w)$.  That (\ref{zeropdf1}) and (\ref{zeropdf2}) are equal follows since $W$ is symmetric.  Formulas  (\ref{zeropdf1}) and (\ref{zeropdf2}) are certainly true for $w\geq 0$ and $w<0$, respectively.  But (\ref{zeropdf1}) and (\ref{zeropdf2}) are equal and so both formulas must hold for all $w\in\mathbb{R}$.  Finally, from (\ref{zeropdf1}) and (\ref{zeropdf2}) we have $f_{W^{*(n)}}(w)=f_{W^{*(n)}}(-w)$, hence $W^{*(n)}$ is also symmetric.

(iii) Substitute $w^{p+1}/(p+1)$ and $w^{p+1}\mathrm{sgn}(w)/(p+1)$ for $f(w)$ in the characterising equation (\ref{islip}).

(iv) Let $U_1,\ldots, U_n$ and $V_n$ be defined as in Proposition \ref{doublesquare}.  Then, from Propositions \ref{zerobox} and \ref{doublesquare}, we have
\[W^{*(n)}\stackrel{\mathcal{D}}{=}V_nW^{\square}\stackrel{\mathcal{D}}{=}V_n\prod_{k=1}^nW_k^{\square}=\prod_{k=1}^nU_kW_k^{\square}\stackrel{\mathcal{D}}{=}\prod_{k=1}^nW_k^*.\]

(v) Let $g$ be a function such that $\mathbb{E}Wg(W)$ exists, and define $\tilde{g}(x)=cg(cx)$.  Then $\tilde{g}^{(k)}(x)=c^{k+1}g^{(k)}(cx)$.  As $W^{*(n)}$ has the $W$-zero bias distribution of order $n$,
\[\mathbb{E}cWg(cW)=\mathbb{E}W\tilde{g}(W)=\sigma^2\mathbb{E}A_n\tilde{g}(W^{*(n)})=(c\sigma)^2\mathbb{E}A_ng(cW^{*(n)}).\]
Hence $cW^{*(n)}$ has the $cW$-zero bias distribution of order $n$.
\end{proof}

\section{Zero bias approach bounds for the product of two independent normal distributions}

We now illustrate how the product normal Stein equation and zero bias transformation may be used together to assess distributional distances for statistics that are asymptotically distributed as the product of two independent central normal random variables.  Theorem \ref{jazzz}, which is a natural generalisation of Theorem 3.1 of Goldstein and Reinert \cite{goldstein}, shows how the distance between an arbitrary mean zero, finite variance random variable $W$ and a product normal random variable with the same variance can be bounded by the distance between $W$ and a variate $W^{*(n)}$ with the $W$-zero biased distribution of order $n$ defined on a joint space.  The bounds in Theorem \ref{jazzz} only hold if $(A_nf)'(x)$ and $(A_nf)''(x)$ exist and are bounded.  Of course, this might not be the case when $n\geq 3$. 

\begin{theorem} \label{jazzz} Let $W$ be a mean zero random variable with variance $\sigma^2$.  Suppose that $(W,W^{*(n)})$ is given on a joint probability space so that $W^{*(n)}$ has the $W$-zero biased distribution of order $n$.  Then 
\begin{equation}\label{zezozr}|\mathbb{E}h(W)-\mathrm{\mathrm{PN}}_n^{\sigma^2}h| \leq\sigma^2\|(A_nf)'\|\mathbb{E}|W-W^{*(n)}|,
\end{equation}
where $f$ is the solution of the $\mathrm{\mathrm{PN}}(n,\sigma^2)$ Stein equation (\ref{delta}).

Suppose now that $W=\prod_{k=1}^nW_k$, where the $W_k$ are independent.  Then 
\begin{align}|\mathbb{E}h(W)-\mathrm{\mathrm{PN}}_n^{\sigma^2}h|&\leq\sigma^2\|(A_nf)'\|\sqrt{\mathbb{E}[\mathbb{E}(W-W^{*(n)}| W_1,\ldots,W_n)^2]}\nonumber\\
&\quad+\frac{\sigma^2}{2}\|(A_nf)''\|\mathbb{E}(W-W^{*(n)})^2.
\end{align}
The quantities $\|(A_2f)'\|$ and $\|(A_2f)''\|$ are bounded in Lemma \ref{arflem}.
\end{theorem}

\begin{proof}We prove the second bound; the first bound is obtained by a similar but simpler calculation.  Using equation (\ref{delta}), we have
\begin{equation*}\mathbb{E}h(W)-\mathrm{\mathrm{PN}}_n^{\sigma^2}h=\sigma^2[A_nf(W)-Wf(W)]=\sigma^2\mathbb{E}[A_nf(W)-A_nf(W^{*(n)})].
\end{equation*}
By Taylor expansion, we have 
\begin{equation*}|\mathbb{E}h(W)-\mathrm{\mathrm{PN}}_n^{\sigma^2}h|\leq\sigma^2|\mathbb{E}[(W-W^{*(n)})(A_nf)'(W)]|+\frac{\sigma^2}{2}\|(A_nf)''\|\mathbb{E}(W-W^{*(n)})^2.
\end{equation*}
For the first term, we condition on $W_1,\ldots,W_n$ and then apply the Cauchy-Schwarz inequality to obtain
\begin{align*}|\mathbb{E}[(W-W^{*(n)})(A_nf)'(W)]|&\leq\mathbb{E}[(A_nf)'(W)\mathbb{E}[W-W^{*(n)}|W_1,\ldots,W_n]] \\
&\leq \|(A_nf)'\|\sqrt{\mathbb{E}[\mathbb{E}(W-W^{*(n)}| W_1,\ldots,W_n)^2]},
\end{align*}
which yields the desired bound.
\end{proof}

We now use Theorem \ref{jazzz} to obtain two bounds for the error in approximating a statistic that has an asymptotic $\mathrm{\mathrm{PN}}(2,1)$ distribution by its limiting distribution.    

\begin{corollary} \label{halifax stupid} Suppose $X,X_1,\ldots,X_m$, $Y,Y_1,\ldots,Y_n$ are independent random variables with zero mean, unit variance and bounded absolute third moment, with $X_i \stackrel{\mathcal{D}}{=}
X$ for all $i=1,\ldots, m$ and $Y_j \stackrel{\mathcal{D}}{=} Y$ for all $j=1,\ldots,n$.  Let $W_1=\sum_{i=1}^mX_i$, $W_2=\sum_{j=1}^nY_j$ and set $W=\frac{1}{\sqrt{mn}}W_1W_2$.  Then, for $h\in C_b^1(\mathbb{R})$, we have
\[|\mathbb{E}h(W)-\mathrm{\mathrm{PN}}_2^1h|\leq\frac{13}{8}\bigg(\frac{1}{\sqrt{m}}+\frac{1}{\sqrt{n}}\bigg)\bigg[\|h'\|+\frac{9}{2}\|h-\mathrm{\mathrm{PN}}_2^1h\|\bigg]\mathbb{E}|X|^3|Y|^3.\]
\end{corollary}

\begin{proof}We will apply bound (\ref{zezozr}) of Theorem \ref{jazzz}, and so we just need to bound $\mathbb{E}|W-W^{*(2)}|$.  By part (iv) of Proposition \ref{john hates} and Lemma \ref{construct2}, we have that $W^{*(2)}=\frac{1}{\sqrt{mn}}W_1^*W_2^*$ where $W_1^*=W_1-X_I+X_I^*$ and $W_2^*=W_2-Y_J+Y_J^*$.  By the independence of the collections $X_1,\ldots,X_m$ and $Y_1,\ldots,Y_n$, we have
\begin{align*}\mathbb{E}|W-W^{*(2)}|&=\frac{1}{\sqrt{mn}}\mathbb{E}|(X_I-X_I^*)W_2+(Y_J-Y_J^*)W_1-(X_I-X_I^*)(Y_J-Y_J^*)| \\
&\leq\frac{1}{\sqrt{mn}}\{(\mathbb{E}|X_I|+\mathbb{E}|X_I^*|)\mathbb{E}|W_2|+(\mathbb{E}|Y_J|+\mathbb{E}|Y_J^*|)\mathbb{E}|W_1| \\
&\quad+(\mathbb{E}|X_I|+\mathbb{E}|X_I^*|)(\mathbb{E}|Y_J|+\mathbb{E}|Y_J^*|)\}.
\end{align*}
By part (iii) of Proposition \ref{john hates}, we have that $\mathbb{E}|X_I|\leq\frac{1}{2}\mathbb{E}|X|^3$.  Using this fact and that $\mathbb{E}|W_1|\leq\sqrt{m}$ and $\mathbb{E}|W_2|\leq\sqrt{n}$ gives
\begin{align*}
\mathbb{E}|W-W^{*(2)}|&\leq\frac{1}{\sqrt{m}}(1+\mathbb{E}|X|^3)+\frac{1}{\sqrt{n}}(1+\mathbb{E}|Y|^3)+\frac{1}{\sqrt{mn}}(1+\mathbb{E}|X|^3)(1+\mathbb{E}|Y|^3) \\
&\leq\frac{3\mathbb{E}|X|^3}{2\sqrt{m}}+\frac{3\mathbb{E}|Y|^3}{2\sqrt{n}}+\frac{9\mathbb{E}|X|^3\mathbb{E}|Y|^3}{4\sqrt{mn}}\leq\frac{13}{8}\bigg(\frac{1}{\sqrt{m}}+\frac{1}{\sqrt{n}}\bigg)\mathbb{E}|X|^3\mathbb{E}|Y|^3,
\end{align*}
where we used the inequalities $\mathbb{E}|X|^3\geq 1$ and $(mn)^{-1/2}\leq\frac{1}{2}(m^{-1/2}+n^{-1/2})$ to simplify the bound.  Using Lemma \ref{arflem} to bound $\|(A_2f)'\|$ completes the proof.
\end{proof}

When $\mathbb{E}X^3=\mathbb{E}Y^3=0$ we can use the second bound of Theorem \ref{jazzz} to obtain a bound on the rate of convergence of $W$ to its limiting distribution that is of order $m^{-1}+n^{-1}$ for smooth test functions.  This approach is similar to the one used by Goldstein and Reinert \cite{goldstein}, who used a zero bias coupling approach to obtain a bound of order $n^{-1}$, for smooth test functions, for normal approximation under the assumption that $\mathbb{E}X^3=0$.   

\begin{corollary} \label{vo2} Let the $X_i$, $Y_j$ and $W$ be defined as in Corollary \ref{halifax stupid}, but with the extra condition that $\mathbb{E}X^3=\mathbb{E}Y^3=0$, and $\mathbb{E}X^4$, $\mathbb{E}Y^4<\infty$.  Then, for $h\in C_b^2(\mathbb{R})$, we have
\begin{align}\label{vo270}|\mathbb{E}h(W)-\mathrm{\mathrm{PN}}_2^1h| &\leq\bigg(\frac{1}{m}+\frac{1}{n}\bigg)\bigg[\frac{7}{2}\|h''\|+\|h'\|+\frac{123}{4}\|h-\mathrm{\mathrm{PN}}_2^1h\|\bigg]\mathbb{E}X^4\mathbb{E}Y^4.
\end{align}
\end{corollary}

\begin{proof}We make use of the second bound in Theorem \ref{jazzz}, and so require bounds on the quantities $\sqrt{\mathbb{E}[\mathbb{E}(W-W^{*(2)}|W_1,W_2)^2]}$ and $\mathbb{E}(W-W^{*(2)})^2$.  We have
\begin{align*}&\mathbb{E}[W-W^{*(2)}|W_1,W_2]\\
&=\frac{1}{\sqrt{mn}}\mathbb{E}[W_1W_2-W_1^*W_2^*|W_1,W_2] \\
&=\frac{1}{\sqrt{mn}}\mathbb{E}[(X_I-X_I^*)W_2+(Y_J-Y_J^*)W_1 -(X_I-X_I^*)(Y_J-Y_J^*)|W_1,W_2] \\
&=\frac{1}{\sqrt{mn}}\{W_2\mathbb{E}[X_I|W_1]+W_1\mathbb{E}[Y_J|W_2]-\mathbb{E}[X_I|W_1]\mathbb{E}[Y_J|W_2]\} \\
&=\frac{1}{\sqrt{mn}}\bigg(\frac{1}{m}+\frac{1}{n}-\frac{1}{mn}\bigg)W_1W_2,
\end{align*}
where we used that $X_I^*$ and $W_1$ are independent and from part (iii) of Proposition \ref{john hates} that $\mathbb{E}X_I^*=\frac{1}{2}\mathbb{E}X^3=0$ to obtain the third equality, and that $\mathbb{E}(X_I|W_1)=\frac{1}{m}W_1$ to obtain the final equality.  As $\mathbb{E}W_1^2=m$ and $\mathbb{E}W_2^2=n$, we have
\[\sqrt{\mathbb{E}[\mathbb{E}(W-W^{*(2)}|W_1,W_2)^2]}=\frac{1}{m}+\frac{1}{n}-\frac{1}{mn} 
<\frac{1}{m}+\frac{1}{n}.\]
We now bound the second term:
\begin{align*}\mathbb{E}(W-W^{*(2)})^2&=\frac{1}{mn}\mathbb{E}[\{(X_I-X_I^*)W_2+(Y_J-Y_J^*)W_1-(X_I-X_I^*)(Y_J-Y_J^*)\}^2] \\
&\leq\frac{3}{mn}[\mathbb{E}(X_I-X_I^*)^2\mathbb{E}W_2^2+\mathbb{E}(Y_J-Y_J^*)^2\mathbb{E}W_1^2 \\
&\quad+\mathbb{E}(X_I-X_I^*)^2\mathbb{E}(Y_J-Y_J^*)^2],
\end{align*}
where we used that $(a+b+c)^2\leq 3(a^2+b^2+c^2)$ to obtain the inequality.  By part (iii) of Proposition \ref{john hates},  
\[\mathbb{E}(X_I-X_I^*)^2=\mathbb{E}X_I^2+\mathbb{E}(X_I^*)^2=1+\frac{1}{3}\mathbb{E}X^4\leq\frac{4}{3}\mathbb{E}X^4,\]
and therefore
\begin{align*}\mathbb{E}(W-W^{*(2)})^2&\leq\frac{4\mathbb{E}X^4}{m}+\frac{4\mathbb{E}Y^4}{n}+\frac{16\mathbb{E}X^4\mathbb{E}Y^4}{3mn} \leq\bigg(4+\frac{8}{3}\bigg)\bigg(\frac{1}{m}+\frac{1}{n}\bigg)\mathbb{E}X^4\mathbb{E}Y^4 \\
&<7\bigg(\frac{1}{m}+\frac{1}{n}\bigg)\mathbb{E}X^4\mathbb{E}Y^4,
\end{align*} 
where we used the inequalities $\mathbb{E}X^4\geq 1$ and $(mn)^{-1}\leq\frac{1}{2}(m^{-1}+n^{-1})$ to obtain the second inequality.  Applying Lemma \ref{arflem} to bound $\|(A_2f)'\|$ and $\|(A_2f)''\|$ and then using the inequality $\mathbb{E}X^4\geq 1$ to simplify the resulting bound yields (\ref{vo270}).
\end{proof}

\section{Bounds for a general $n$ via the multivariate normal Stein equation}

In the previous section, we saw that the product normal Stein equation and the zero bias transformation of order $n$ can be applied together to derive bounds for product normal approximation, provided we have bounds for the appropriate lower order derivatives of the solution of the Stein equation.  However, in this paper, we have only been able to achieve this for the case of $n=2$ products.  A similar problem was in encountered by Arras et al$.$ \cite{aaps16}, in which an $n$-th order Stein equation was obtained for a general linear combination of $n$ independent gamma random variables.  They were also unable to bound the appropriate lower order derivatives of the solution, but were still able to prove approximation theorems by bypassing the Stein equation (see Section 3 of \cite{aaps16}).  

In this section, we take the same philosophy as \cite{aaps16} and, by bypassing the product normal Stein equation, we are able to prove product normal approximation theorems for general $n$.  In a recent paper, Gaunt \cite{gaunt normal} introduced a general method for proving approximation theorems in which the limit distribution can be represented as a functions of multivariate normal random variables (see also Gaunt et al$.$ \cite{gaunt chi square}, in which the technique is applied for the case of a chi-square limit). 

Let $X_{1,1},\ldots,X_{n,1},\ldots,X_{1,d},\ldots,X_{n,d}$ be independent random variables with mean zero and unit variance.  Define $W_j=\frac{1}{\sqrt{n}}\sum_{i=1}^nX_{ij}$ and let $\mathbf{W}=(W_1,\ldots,W_d)^T$, which, by the central limit theorem, converges to the standard $d$-dimensional multivariate normal random variable $\mathbf{Z}$.  In \cite{gaunt normal}, general bounds were given for distributional distance between $g(\mathbf{W})$ and $g(\mathbf{Z})$, where $g:\mathbb{R}^d\rightarrow\mathbb{R}$ satisfies certain differentiability and growth rate conditions.  The approach bypasses the Stein equation for the limit distribution $g(\mathbf{Z})$ by using the multivariate normal Stein equation (see Goldstein and Rinott \cite{goldstein1}) to approximate the random vector $\mathbf{W}$ and then links this approximation to one for $g(\mathbf{W})$ by a suitably chosen test function.  
This approach allows a large class of limit distributions to be treated within one framework, including products of independent normal random variables.  We now state general bounds (Theorems 3.2 and 3.4 of Gaunt \cite{gaunt normal}), which we shall then apply to obtain some $\mathrm{PN}(n,1)$ approximation theorems that hold for general $n\geq1$.  

\begin{theorem}\label{winfirst1}(\cite{gaunt normal}, Theorem 3.2)  Let $X_{1,1},\ldots,X_{n,1},\ldots,X_{1,d},\ldots,X_{n,d}$ be independent random variables with $\mathbb{E}X_{ij}^k=\mathbb{E}Z^k$ for all $1\leq i\leq n_j$, $1\leq j\leq d$, $1\leq k\leq p$, where $Z\sim N(0,1)$.
Suppose also that $\mathbb{E}|X_{ij}|^{r_l+p+1}<\infty$ for all $i$, $j$ and $l$.  Let $P(\mathbf{w})=A+B\sum_{i=1}^d|w_i|^{r_i}$, where $A$, $B$ and $r_1,\ldots,r_d$ are non-negative constants.  Suppose $g\in C^p(\mathbb{R}^d)$ is such that $\big|\frac{\partial^kg(\mathbf{w})}{\partial w_j^k}\big|^{p/k}\leq P(\mathbf{w})$ for all $1\leq j\leq d$, $1\leq k\leq p$.  Then, for $h\in C_b^p(\mathbb{R})$, 
\begin{align*} &|\mathbb{E}h(g(\mathbf{W}))-\mathbb{E}h(g(\mathbf{Z}))|\leq\frac{p+1}{p!}h_p\sum_{j=1}^d\sum_{i=1}^{n_j}\frac{1}{n_j^{(p+1)/2}}\bigg[A\mathbb{E}|X_{ij}|^{p+1}\\
&\quad+B\sum_{k=1}^d2^{r_k}\bigg(2^{r_k}\mathbb{E}|X_{ij}|^{p+1}\mathbb{E}|W_j|^{r_k} +\frac{2^{r_k}}{n_j^{r_k/2}}\mathbb{E}|X_{ij}|^{r_k+p+1}+\mathbb{E}|Z|^{r_k+1}\mathbb{E}|X_{ij}|^{p+1}\bigg)\bigg],
\end{align*}
where $h_p=\sum_{k=1}^p{p\brace k}\|h^{(k)}\|$ and the Stirling number of the second kind are given by ${p\brace k}=\frac{1}{k!}\sum_{j=0}^k(-1)^{k-j}\binom{k}{j}j^p$ (see Olver et al$.$ \cite{olver}).
\end{theorem}

\begin{theorem}\label{multieveng}(\cite{gaunt normal}, Theorem 3.4)   Let $X_{1,1},\ldots,X_{n,1},\ldots,X_{1,d},\ldots,X_{n,d}$ be independent random variables with $\mathbb{E}X_{ij}^k=\mathbb{E}Z^k$ for all $1\leq i\leq n_j$, $1\leq j\leq d$,  $1\leq k\leq p$.  Suppose also that $\mathbb{E}|X_{ij}|^{r_l+p+2}<\infty$ for all $i$, $j$ and $l$.  Let $P(\mathbf{w})=A+B\sum_{i=1}^d|w_i|^{r_i}$, where $A$, $B$ and $r_1,\ldots,r_d$ are non-negative constants.  Suppose $g\in C^{p+2}(\mathbb{R}^d)$ is such that $\big|\frac{\partial^kg(\mathbf{w})}{\partial w_j^k}\big|^{(p+2)/k}\leq P(\mathbf{w})$ for all $1\leq j\leq d$, $1\leq k\leq p+2$.  Suppose further that $g$ is an even function ($g(\mathbf{w})=g(-\mathbf{w})$ for all $\mathbf{w}\in\mathbb{R}^d$).  Then, for $h\in C_b^{p+2}(\mathbb{R})$, 
\begin{align*}&|\mathbb{E}h(g(\mathbf{W}))-\mathbb{E}h(g(\mathbf{Z}))| 
\leq\frac{h_{p+2}}{p!}\bigg\{\frac{1}{p+2}\sum_{j=1}^d\sum_{i=1}^{n_j}\frac{1}{n_j^{p/2+1}}\bigg(\frac{p+2}{p+1}+|\mathbb{E}X_{ij}^{p+1}|\bigg)\bigg[A\mathbb{E}|X_{ij}|^{p+2} \\
&\quad+B\sum_{k=1}^d2^{r_k}\bigg(2^{r_k}\mathbb{E}|X_{ij}|^{p+2}\mathbb{E}|W_j|^{r_k}+\frac{2^{r_k}}{n_j^{r_k/2}}\mathbb{E}|X_{ij}|^{r_k+p+2}+\mathbb{E}|Z|^{r_k}\mathbb{E}|X_{ij}|^{p+2}\bigg)\bigg]\\
&\quad  +\frac{3}{2}\sum_{j=1}^d\sum_{i=1}^{n_j}\frac{|\mathbb{E}X_{ij}^{p+1}|}{n_j^{(p+1)/2}}\sum_{k=1}^d\sum_{l=1}^{n_k}\frac{1}{n_k^{3/2}}\bigg[A\mathbb{E}|X_{ij}|^3 \\
&\quad+B\sum_{k=1}^d\!3^{r_k}\bigg(\!2^{r_k}\mathbb{E}|X_{ij}|^3\mathbb{E}|W_j|^{r_k}\!+\!\frac{2^{r_k}}{n_j^{r_k/2}}\mathbb{E}|X_{ij}|^{r_k+3}\!+\!2\mathbb{E}|Z|^{r_k+1}\mathbb{E}|X_{ij}|^{3}\!\bigg)\bigg]\bigg\}.
\end{align*}
\end{theorem}

A $\mathrm{PN}(d,1)$ random variable can be represented as $g(\mathbf{Z})$, where $g(\mathbf{w})=\prod_{j=1}^dw_j$ is infinitely often differentiable and has derivatives of polynomial growth as $|\mathbf{w}|\rightarrow\infty$.  Also, when $d$ is even, the function $g$ is even.  We can therefore apply Theorems \ref{winfirst1} and \ref{multieveng} to obtain the following bounds for product normal approximation.

\begin{corollary}\label{cor5}Fix $d>1$.  Let $X,X_{1,1},\ldots,X_{n,1},\ldots,X_{1,d},\ldots,X_{n,d}$ be i.i.d$.$ random variables with $\mathbb{E}X^k=\mathbb{E}Z^k$ for all  $1\leq k\leq p$, and such that $\mathbb{E}|X|^{2p+1}<\infty$.  Define $W=\prod_{j=1}^dW_j$, where $W_j=\frac{1}{\sqrt{n_j}}\sum_{i=1}^{n_j}X_{ij}$.  Then, for $h\in C_b^p(\mathbb{R})$,
\begin{align}|\mathbb{E}h(W)-\mathrm{\mathrm{PN}}_d^1h|&\leq\frac{2^p(p+1)}{p!}h_p\sum_{j=1}^d\frac{1}{n_j^{(p-1)/2}}\bigg[2^p\mathbb{E}|X|^{p+1}\mathbb{E}|W_1|^p+\mathbb{E}|Z|^{p+1}\mathbb{E}|X|^{p+1}\nonumber\\
\label{corbound51}&\quad+\frac{2^p}{n_j^{p/2}}\mathbb{E}|X|^{2p+1}\bigg].
\end{align}
\end{corollary}

\begin{proof}We apply Theorem \ref{winfirst1} with $g(\mathbf{w})=\prod_{j=1}^dw_j$.  For any $j=1,\ldots,d$ we have that $\frac{\partial g(\mathbf{w})}{\partial w_j}=\prod_{l\not=j}^dw_l$ and $\frac{\partial^2 g(\mathbf{w})}{\partial w_j^2}=0$.  Therefore, we can take $P(\mathbf{w})=\frac{1}{d}\sum_{j=1}^d|w_j|^p$.  Applying the bound of  Theorem \ref{winfirst1} with $A=0$, $B=\frac{1}{d}$ and $r_1=\cdots=r_d=p$ yields (\ref{corbound51}). 
\end{proof}

\begin{corollary}\label{cor6}Let $d\geq2$ be even.  Let $X,X_{1,1},\ldots,X_{n,1},\ldots,X_{1,d},\ldots,X_{n,d}$ be i.i.d$.$ random variables with $\mathbb{E}X^k=\mathbb{E}Z^k$ for all  $1\leq k\leq p$, and such that $\mathbb{E}|X|^{2p+4}<\infty$.  Let $W$ be defined as in Corollary \ref{cor5}.  Then, for $h\in C_b^{p+2}(\mathbb{R})$,
\begin{align}|\mathbb{E}h(W)-\mathrm{\mathrm{PN}}_d^1h|&\leq\frac{h_{p+2}}{p!}\bigg\{\frac{2^{p+2}}{p+2}\sum_{j=1}^d\frac{1}{n_j^{p/2}}\bigg(\frac{p+2}{p+1}+|\mathbb{E}X^{p+1}|\bigg)\bigg[2^{p+2}\mathbb{E}|X|^{p+2}\mathbb{E}|W_1|^{p+2}\nonumber\\
&\quad+\mathbb{E}|Z|^{p+2}\mathbb{E}|X|^{p+2}+\frac{2^{p+2}}{n_j^{p/2+1}}\mathbb{E}|X|^{2p+4}\bigg] \nonumber\\
&\quad+\frac{3^{p+3}}{2}|\mathbb{E}X^{p+1}|\sum_{j,k=1}^d\frac{1}{n_j^{(p-1)/2} n_k^{1/2}}\bigg[2^{p+2}\mathbb{E}|X|^3\mathbb{E}|W_1|^{p+2}\nonumber\\
\label{corbound52}&\quad+2\mathbb{E}|Z|^{p+3}\mathbb{E}|X|^3+\frac{2^{p+2}}{n_j^{p/2+1}}\mathbb{E}|X|^{p+5}\bigg]\bigg\}.
\end{align}
\end{corollary}

\begin{proof}Here $d$ is even, so we apply Theorem \ref{multieveng}.  Arguing as in the proof of Corollary \ref{cor5}, we take $P(\mathbf{w})=\frac{1}{d}\sum_{j=1}^d|w_j|^{p+2}$.  On applying the bound of  Theorem \ref{winfirst1} with $A=0$, $B=\frac{1}{d}$ and $r_1=\cdots=r_d=p+2$ we obtain (\ref{corbound52}).
\end{proof}

\begin{remark}\emph{From Corollary \ref{cor5}, we have a bound on the rate of convergence of $W$ to the $\mathrm{PN}(d,1)$ of order $n_1^{-(p-1)/2}+\cdots+n_d^{-(p-1)/2}$ for smooth test functions, provided that first $p$ moments of $X$ and the standard normal distribution agree.  From Corollary \ref{cor6}, we see that, for even $d$, in which case $g(\mathbf{w})=\prod_{j=1}^dw_j$ is an even function, this rate of convergence improves to order  $n_1^{-p/2}+\cdots+n_d^{-p/2}$.} 
\end{remark}

\begin{remark}\emph{It is instructive to compare the bound of Corollary \ref{vo2} that was obtained using the zero bias coupling approach and the bounds of Corollaries \ref{cor5} and \ref{cor6} (in the case $d=2$) that were obtained by bypassing the product normal Stein equation.  The bound of Corollary \ref{vo2} is of order $n_1^{-1}+n_2^{-1}$ and was derived under the assumption that the $X_i$ and $Y_i$ had third moments equal to zero.  Applying (\ref{corbound51}) with this setup gives a bound of the same order, although we must impose stronger moment assumptions (bounded absolute seventh moment, rather than bounded fourth moment) and stronger conditions on the test functions (we require $h\in C_b^3(\mathbb{R})$, instead of $h\in C_b^2(\mathbb{R})$). Whilst the bound of Corollary \ref{vo2} performs better in this example, the proof relies heavily on the assumption of third moments equal to zero.  However, even when the third moments are non-zero, we can obtain an $O(n_1^{-1}+n_2^{-1})$ bound from (\ref{corbound52}), under the assumption of bounded eighth moments are tests functions from the class $C_b^4(\mathbb{R})$.}
\end{remark}

\appendix

\section{Proof of Theorem \ref{compoz7}} 
We begin by obtaining formulas for the the first four derivatives of the solution (\ref{ink}).
\begin{lemma} \label{pudding} Suppose $h \in C_b^{3}(\mathbb{R})$ and let $\tilde{h}(x)=h(x)-\mathrm{\mathrm{PN}}_2^1h$.  Then the first four derivatives of the solution (\ref{ink}) of the $\mathrm{\mathrm{PN}}(2,1)$ Stein equation (\ref{deltappp}), in the region $x> 0$, are given by
\begin{eqnarray*}f'(x) &=& -K_0'(x) \int_0^x  I_0(y)\tilde{h}(y)\,\mathrm{d}y -I_0'(x) \int_x^{\infty}  K_0(y)\tilde{h}(y)\,\mathrm{d}y, \\
f''(x) &=& \frac{\tilde{h}(x)}{x} -K_0''(x) \int_0^x  I_0(y)\tilde{h}(y)\,\mathrm{d}y  -I_0''(x) \int_x^{\infty} K_0(y)\tilde{h}(y)\,\mathrm{d}y, \\
f^{(3)}(x) &=&\frac{h'(x)}{x}-\frac{2\tilde{h}(x)}{x^2} -K_0^{(3)}(x) \int_0^x  I_0(y)\tilde{h}(y)\,\mathrm{d}y -I_0^{(3)}(x) \int_x^{\infty}  K_0(y)\tilde{h}(y)\,\mathrm{d}y, \\
f^{(4)}(x) &=& \frac{h''(x)}{x}-\frac{3h'(x)}{x^2} +\bigg(\frac{6}{x^3}+\frac{1}{x}\bigg)\tilde{h}(x)-K_0^{(4)}(x) \int_0^x  I_0(y)\tilde{h}(y)\,\mathrm{d}y\\
&&  -I_0^{(4)}(x) \int_x^{\infty}  K_0(y)\tilde{h}(y)\,\mathrm{d}y.
\end{eqnarray*}
\end{lemma}

\begin{proof}We will make repeated us of the Leibniz's theorem for differentiation of an integral, which states that provided the functions $u(y,x)$ and $\frac{\partial u}{\partial x}(y,x)$ are continuous in both $x$ and $y$ in the region $a(x)\leq y\leq b(x)$, $x_0\leq x\leq x_1$, and the functions $a(x)$ and $b(x)$ are continuous and have continuous derivatives for $x_0\leq x\leq x_1$, then for $x_0\leq x\leq x_1$,
\begin{equation} \label{mega} \frac{\mathrm{d}}{\mathrm{d}x}\int_{a(x)}^{b(x)}u(y,x)\,\mathrm{d}y =\int_{a(x)}^{b(x)}\frac{\partial}{\partial x}u(y,x)\,\mathrm{d}y+u(b,x)\frac{\mathrm{d}b}{\mathrm{d}x}-u(a,x)\frac{\mathrm{d}a}{\mathrm{d}x}.
\end{equation} 
We will also make use of the identity 
\begin{equation}\label{wront} I_{0}(x)K_{1}(x)+I_{1}(x)K_{0}(x)=\frac{1}{x}
\end{equation}
(see Olver et al$.$ \cite{olver}), as well as the formulas (\ref{diffIii}) -- (\ref{111213}) for the first three derivatives of $I_0(x)$ and $K_0(x)$.

It easy to compute the first and second derivatives by applying (\ref{mega}), the differentiation formulas (\ref{diffIii}) and (\ref{diffKii}) and identity (\ref{wront}).  The calculation of the third derivative is still straightforward but a little longer.  We differentiate the formula for the second derivative using (\ref{mega}) to obtain
\begin{align*}f^{(3)}(x)&=\frac{h'(x)}{x}-\frac{\tilde{h}(x)}{x^2} -K_0^{(3)}(x) \int_0^x  I_0(y)\tilde{h}(y)\,\mathrm{d}y \\
& \quad-I_0^{(3)}(x) \int_x^{\infty}  K_0(y)\tilde{h}(y)\,\mathrm{d}y +\tilde{h}(x)[-I_0(x)K_0''(x)+K_0(x)I_0''(x)].
\end{align*}
Using the differentiation formulas (\ref{ponmu}) and (\ref{thmcor}) and identity (\ref{wront}) allows us to calculate the term in the brackets ($*$) from the above expression
\begin{align*}(*)&=-x^{\nu}I_0(x)K_0''(x) +x^{\nu}K_0(x)I_0''(x) =-\frac{1}{x}[I_0(x)K_{1}(x)+I_{1}(x)K_0(x)] =\frac{1}{x^2}.
\end{align*}
Substituting ($*$) into the expression for $f^{(3)}(x)$ gives the result.

Finally, we verify the formula for the fourth derivative.  We differentiate the formula for the third derivative using (\ref{mega}) to obtain
\begin{align*}f^{(4)}(x)&=\frac{h''(x)}{x}-\frac{3h'(x)}{x^2}+\frac{4\tilde{h}(x)}{x^3}  -K_0^{(4)}(x) \int_0^x  I_0(y)\tilde{h}(y)\,\mathrm{d}y \\
&\quad-I_0^{(4)}(x) \int_x^{\infty}  K_0(y)\tilde{h}(y)\,\mathrm{d}y +\tilde{h}(x)[-I_0(x)K_0^{(3)}(x)+K_0(x)I_0^{(3)}(x)].
\end{align*}
Using the differentiation formulas (\ref{mjyt}) and (\ref{111213}) and identity (\ref{wront}) allows us to calculate the term in the brackets ($**$) from the above expression
\begin{align*}(**)&=-x^{\nu}I_0(x)K_0^{(3)}(x)+x^{\nu}K_0(x)I_0^{(3)}(x) \\
&=\bigg(\frac{2}{x^2}+1\bigg)[I_0(x)K_{1}(x)+I_{1}(x)K_0(x)] =\frac{2}{x^3}+\frac{1}{x}.
\end{align*}
Substituting ($**$) into the expression for $f^{(4)}(x)$ gives the result.
\end{proof}
In their current forms the derivatives of the solution are not suitable for bounding, as they contain terms that are singular.  In the next lemma we use integration by parts to group the singularities together and then apply the triangle inequality.  

The following notation for the repeated integral of the function $I_0(x)$ will be used in the next lemma.  It is consistent with the notation in Gaunt \cite{gaunt}.
\begin{equation} \label{super1} I_{(0,0,0)}(x)=I_0(x), \quad I_{(0,0,n)}(x)=\int_0^xI_{(0,0,n-1)}(y)\,\mathrm{d}y,\:\: n =1,2,3\ldots.
\end{equation}

\begin{lemma} \label{compoz} Suppose $h \in C_b^3(\mathbb{R})$ and let $\tilde{h}(x)=h(x)-\mathrm{\mathrm{PN}}_2^1h$.  Then the solution (\ref{ink}) of the $\mathrm{\mathrm{PN}}(2,1)$ Stein equation (\ref{deltappp}) and its first four derivatives, in the region $x> 0$, may be be bounded as follows
\begin{eqnarray*}|f(x)| &\leq&\|\tilde{h}\|\left|I_{(0,0,1)}(x)K_0(x) \right| +\|\tilde{h}\|\bigg|I_0(x) \int_x^{\infty}  K_0(y)\,\mathrm{d}y\bigg|,  \\
\label{foggyz2}|f'(x)| &\leq&\|\tilde{h}\|\left|I_{(0,0,1)}(x)K_0'(x)\right| +\|\tilde{h}\|\bigg|I_0'(x) \int_x^{\infty}  K_0(y)\,\mathrm{d}y\bigg|,  \nonumber \\
|f''(x)| &\leq&\|\tilde{h}\|\bigg|\frac{1}{x} -I_{(0,0,1)}(x)K_0''(x)\bigg| +\|h'\|\left|I_{(0,0,2)}(x)K_0''(x)\right| \nonumber\\
&&+\|\tilde{h}\|\bigg|I_0''(x) \int_x^{\infty}  K_0(y)\,\mathrm{d}y\bigg|, \nonumber \\
|f^{(3)}(x)| &\leq&\|\tilde{h}\|\bigg|\frac{2}{x^2}+I_{(0,0,1)}(x)K_0^{(3)}(x)\bigg| +\|h'\|\bigg|\frac{1}{x}+I_{(0,0,2)}(x)K_0^{(3)}(x)\bigg| \\
&& +\|h''\|\left|I_{(0,0,3)}(x)K_0^{(3)}(x)\right| +\|\tilde{h}\|\bigg|I_0^{(3)}(x) \int_x^{\infty}  K_0(y)\,\mathrm{d}y\bigg|, \nonumber 
\end{eqnarray*}
\begin{eqnarray*}|f^{(4)}(x)| &\leq&\|\tilde{h}\|\bigg|\frac{6}{x^3}+\frac{1}{x}-I_{(0,0,1)}(x)K_0^{(4)}(x)\bigg| +\|h'\|\bigg|\frac{3}{x^2}-I_{(0,0,2)}(x)K_0^{(4)}(x)\bigg| \nonumber  \\
&&+\|h''\|\bigg|\frac{1}{x}-I_{(0,0,3)}(x)K_0^{(4)}(x)\bigg| +\|h^{(3)}\|\left|I_{(0,0,4)}(x)K_0^{(4)}(x)\right| \\
&&+\|\tilde{h}\|\bigg|I_0^{(4)}(x) \int_x^{\infty}  K_0(y)\,\mathrm{d}y\bigg|. \nonumber
\end{eqnarray*}
\end{lemma}

\begin{proof}The first two bounds are immediate from the the formulas for $f(x)$ and $f'(x)$ that are given in Lemma \ref{pudding}.  Integrating by parts and using the notation for the repeated integral of $I_0(x)$, which is defined above, gives
\begin{align*}f''(x)&=\frac{\tilde{h}(x)}{x}-K_0''(x)\bigg[\tilde{h}(x)\int_0^x  I_0(y)\,\mathrm{d}y-\int_0^xh'(y)\bigg(\int_0^y  I_0(u)\,\mathrm{d}u\bigg)\,\mathrm{d}y \bigg] \\
&\quad  -I_0''(x) \int_x^{\infty} K_0(y)\tilde{h}(y)\,\mathrm{d}y \\
&=\tilde{h}(x)\bigg(\frac{1}{x}-I_{(0,0,1)}(x)K_0''(x)\bigg) -K_0''(x)\int_0^x h'(y)I_{(0,0,1)}(y)\,\mathrm{d}y \\
&\quad-I_0''(x) \int_x^{\infty} K_0(y)\tilde{h}(y)\,\mathrm{d}y.
\end{align*}
The bound now follows from the triangle inequality and (\ref{super1}).  The bounds for the third and fourth derivatives are obtained in a similar manner, in which we apply integration by parts to the integrals $I_{(0,0,n)}(x)$. 
\end{proof}

The expressions involving modified Bessel functions that appear in Lemma \ref{compoz} are bounded by Gaunt \cite{gaunt}, and we list these bounds in Appendix C.  We are now in a position to prove Theorem \ref{compoz7}.  Applying the inequalities of Appendix C to the bounds of Lemma \ref{compoz} and a simple change of variables leads to our bounds for the solution of the $\mathrm{\mathrm{PN}}(2,\sigma^2)$ Stein equation.

\vspace{3mm}

\noindent\emph{Proof of Theorem \ref{compoz7}}. Let $\psi_g(x)$ denote the solution (\ref{ink}) of the $\mathrm{\mathrm{PN}}(2,1)$ Stein equation with test function $g$.  Recalling Remark \ref{mark}, it suffices to bound $\psi_g$ and its first four derivatives in the region $x\geq 0$.  Hence, applying the bounds, that are given in Appendix C, for expressions involving modified Bessel functions given in Lemma \ref{compoz} leads to the following bounds on the derivatives of the solution of the $\mathrm{\mathrm{PN}}(2,1)$ Stein equation:
\begin{eqnarray*}\|\psi_g\|&\leq& \bigg(1+\frac{\pi}{2}\bigg)\|g-\mathrm{\mathrm{PN}}_2^1g\|\leq 3\|g-\mathrm{\mathrm{PN}}_2^1g\|, \\
\|\psi_g'\| &\leq& \frac{3}{2}\|g-\mathrm{\mathrm{PN}}_2^1g\|, \\
\|\psi_g''\| &\leq& 2\|g'\|+\bigg(\frac{13}{4}+\frac{\sqrt{\pi}}{2}\bigg)\|g-\mathrm{\mathrm{PN}}_2^1g\|\leq 2\|g'\|+5\|g-\mathrm{\mathrm{PN}}_2^1g\|,\\
\|\psi_g^{(3)}\| &\leq& 4\|g''\|+5\|g'\|+4.89\|g-\mathrm{\mathrm{PN}}_2^1g\|, \\
\|\psi_g^{(4)}\| &\leq& 8\|g^{(3)}\|+9\|g''\|+6.81\|g'\|+15.75\|g-\mathrm{\mathrm{PN}}_2^1g\|, 
\end{eqnarray*}
\begin{eqnarray*}
\|x\psi_g(x)\|&\leq& (0.615+1)\|h-\mathrm{\mathrm{PN}}_2^{1}h\|\leq 2\|h-\mathrm{\mathrm{PN}}_2^{1}h\|, \\
\|x\psi_g'(x)\|&\leq&\frac{3}{2}\|h-\mathrm{\mathrm{PN}}_2^{1}h\|, 
\end{eqnarray*}
where in establishing the bound for $\|\psi_g^{(4)}\|$ we used that $14.61+\frac{1}{4}+\frac{\sqrt{\pi}}{2}<15.75$.  From the $\mathrm{PN}(2,1)$ Stein equation, we have
\[\|x\psi_g''(x)\|\leq \|h-\mathrm{\mathrm{PN}}_2^{1}h\|+\|\psi_g'\|+\|x\psi_g(x)\|\leq\frac{9}{2}\|h-\mathrm{\mathrm{PN}}_2^{1}h\|.\]

We now make a change of variables to obtain bounds for the derivatives of the solution of the $\mathrm{\mathrm{PN}}(2,\sigma^2)$ Stein equation.  The function $f_h(x)=\frac{1}{\sigma}\psi_{g}(\frac{x}{\sigma})$ solves the $\mathrm{\mathrm{PN}}(2,\sigma^2)$ Stein equation 
\[\sigma^2xf''(x)+\sigma^2f'(x)-xf(x)=h(x)-\mathrm{\mathrm{PN}}_2^{\sigma^2} h,\]
where $h(x)=g(\frac{x}{\sigma})$, since $\mathrm{\mathrm{PN}}_2^{\sigma^2} h = \mathrm{\mathrm{PN}}_2^1 g$.  We verify that $\mathrm{\mathrm{PN}}_2^{\sigma^2} h = \mathrm{\mathrm{PN}}_2^1 g$ with the following calculation:
\[\mathrm{\mathrm{PN}}_2^{\sigma^2} h =\int_{-\infty}^{\infty}\frac{1}{\pi\sigma}K_0\bigg(\frac{|x|}{\sigma}\bigg)h(x)\,\mathrm{d}x 
=\int_{-\infty}^{\infty}K_0(|u|)g(u)\,\mathrm{d}u 
= \mathrm{\mathrm{PN}}_2^1 g,
\]
where we made the change of variables $u=\frac{x}{\sigma}$.  We have that $\|f_h^{(k)}\|=\sigma^{-k-1}\|\psi_{g}^{(k)}\|$ and $\|xf_h^{(k)}(x)\|=\sigma^{-k-1}\|x\psi_{g}^{(k)}(x)\|$ for $k\geq 0$, and $\|g-\mathrm{\mathrm{PN}}_2^1g\|=\|h-\mathrm{\mathrm{PN}}_2^{\sigma^2}h\|$ and $\|g^{(k)}\|=\sigma^k\|h^{(k)}\|$ for $k\geq 1$.  This completes the proof of Theorem \ref{compoz7}. \hfill $\square$

\section{Elementary properties of the Meijer $G$-function and modified Bessel functions}

Here we define the Meijer $G$-function and modified Bessel functions and state some of their elementary properties.  For further properties of these functions see, for example, Olver et al$.$ \cite{olver} and Luke \cite{luke}.  The formulas for the Meijer $G$-function are given in in Luke \cite{luke}.  The modified Bessel function formulas can be found in Olver et al$.$ \cite{olver}, except for the second and third order derivative formulas which are given in Gaunt \cite{gaunt confirmation}. 

\subsection{The Meijer $G$-function}

\subsubsection{Definition}

The \emph{Meijer $G$-function} is defined, for $z\in\mathbb{C}\setminus\{0\}$, by the contour integral:
\[G^{m,n}_{p,q}\bigg(z \; \bigg|\; {a_1,\ldots, a_p \atop b_1,\ldots,b_q} \bigg)=\frac{1}{2\pi i}\int_{c-i\infty}^{c+i\infty}z^{-s}\frac{\prod_{j=1}^m\Gamma(s+b_j)\prod_{j=1}^n\Gamma(1-a_j-s)}{\prod_{j=n+1}^p\Gamma(s+a_j)\prod_{j=m+1}^q\Gamma(1-b_j-s)}\,\mathrm{d}s,\]
where $c$ is a real constant defining a Bromwich path separating the poles of $\Gamma(s + b_j)$ from those of $\Gamma(1- a_j- s)$ and where we use the convention that the empty product is $1$.

\subsubsection{Basic properties}
The Meijer $G$-function is symmetric in the parameters $a_1,\ldots,a_n$; $a_{n+1},\ldots,a_p$; $b_1,\ldots,b_m$; and $b_{m+1},\ldots,b_q$.  Thus, if one the $a_j$'s, $j=n+1,\ldots,p$, is equal to one of the $b_k$'s, $k=1,\ldots,m$, the Meijer $G$-function reduces to one of lower order.  For example,
\begin{equation}\label{lukeformula}G_{p,q}^{m,n}\bigg(z \; \bigg| \;{a_1,\ldots,a_{p-1},b_1 \atop b_1\ldots,b_q}\bigg)=G_{p-1,q-1}^{m-1,n}\bigg(z \; \bigg| \;{a_1,\ldots,a_{p-1} \atop b_2\ldots,b_q}\bigg), \quad m,p,q\geq 1.
\end{equation}
The Meijer $G$-function satisfies the identities
\begin{equation}\label{meijergidentity}z^cG_{p,q}^{m,n}\bigg(z \; \bigg| \;{a_1,\ldots,a_p \atop b_1\ldots,b_q}\bigg)=G_{p,q}^{m,n}\bigg(z \; \bigg| \;{a_1+c,\ldots,a_p+c \atop b_1+c\ldots,b_q+c}\bigg), \quad z\in\mathbb{R}
\end{equation}
and
\begin{equation}\label{meijerg-1}G_{p,q}^{m,n}\bigg(z \; \bigg| \;{a_1,\ldots,a_p \atop b_1\ldots,b_q}\bigg)=G_{p,q}^{m,n}\bigg(\frac{1}{z} \; \bigg| \;{1-b_1,\ldots,1-b_q \atop 1-a_1\ldots,1-a_p}\bigg),\quad z\in\mathbb{R}.
\end{equation}

\subsubsection{Integration}
For $\alpha>0$, $\gamma>0$, $a_j<1$ for $j=1,\ldots,n$, and $b_j>-\frac{1}{2}$ for $j=1,\ldots,m$, we have
\begin{equation}\label{meijergintegration}\int_0^{\infty}\cos(\gamma x)G_{p,q}^{m,n}\bigg(\alpha x^2 \; \bigg| \;{a_1,\ldots,a_p \atop b_1\ldots,b_q}\bigg)\,\mathrm{d}x=\sqrt{\pi}\gamma^{-1}G_{p+2,q}^{m,n+1}\bigg(\frac{4\alpha}{\gamma^2} \; \bigg| \;{\frac{1}{2},a_1,\ldots,a_p,0 \atop b_1\ldots,b_q}\bigg).
\end{equation}

\subsubsection{Differential equation}
The Meijer $G$-function satisfies a differential equation of order $\max(p,q)$:
\begin{equation}\label{meijergdiff}\bigg[(-1)^{p-m-n}z\prod_{j=1}^p\bigg(z\frac{\mathrm{d}}{\mathrm{d}z}-a_j+1\bigg)-\prod_{j=1}^q\bigg(z\frac{\mathrm{d}}{\mathrm{d}z}-b_j\bigg)\bigg]G(z)=0.
\end{equation}

\subsection{Modified Bessel functions}

\subsubsection{Definitions}
The \emph{modified Bessel function of the first kind} of order $\nu \in \mathbb{R}$ is defined, for all $x\in\mathbb{R}$, by
\begin{equation*}\label{defI}I_{\nu} (x) = \sum_{k=0}^{\infty} \frac{1}{\Gamma(\nu +k+1) k!} \left( \frac{x}{2} \right)^{\nu +2k}.
\end{equation*}

The \emph{modified Bessel function of the second kind} of order $\nu \in \mathbb{R}$ can be defined in terms of the modified Bessel function of the first kind as follows
\begin{eqnarray*}K_{\nu} (x) &=& \frac{\pi}{2 \sin (\nu \pi)} (I_{-\nu}(x) - I_{\nu} (x)), \quad \nu \not= \mathbb{Z}, \: x \in\mathbb{R}, \\
K_{\nu} (x) &=& \lim_{\mu\to\nu} K_{\mu} (x) = \lim_{\mu\to\nu} \frac{\pi}{2 \sin (\mu \pi)} (I_{-\mu}(x) - I_{\mu} (x)), \quad \nu \in \mathbb{Z}, \: x \in\mathbb{R}.
\end{eqnarray*}

\subsubsection{Basic properties}
For $\nu\in\mathbb{R}$, the modified Bessel function of the first kind $I_{\nu}(x)$ and the modified Bessel function of the second kind $K_{\nu}(x)$ are regular functions of $x$.  For $n\in\mathbb{Z}$, $I_{2n}(x)$ is a real-valued function for all $x\in\mathbb{R}$, with $I_{2n}(-x)=I_{2n}(x)$. The modified Bessel function $I_{2n+1}(x)$ is a real-valued for all $x\in\mathbb{R}$, with $I_{2n+1}(-x)=-I_{2n+1}(x)$.  For $\nu\geq 0$ and $x>0$ we have $I_{\nu}(x)>0$ and $K_{\nu}(x)>0$.  For all $\nu\in\mathbb{R}$, the modified Bessel function $K_{\nu}(x)$ is complex-valued in the region $x<0$.

\subsubsection{Representation in terms of the Meijer $G$-function}
\begin{eqnarray}\label{iomei}I_0(x)&=& G_{0,2}^{1,0}\bigg(-\frac{x^2}{4}\,\bigg|\,0,0\bigg), \quad x\in\mathbb{R}, \\
\label{komei}K_0(|x|) &=&\frac{1}{2}G_{0,2}^{2,0}\bigg( \frac{x^2}{4}\,\bigg|\,0,0\bigg),\quad x\in\mathbb{R}.
\end{eqnarray}

\subsubsection{Asymptotic expansions}
\begin{eqnarray}
\label{Ktend0}K_{0} (x) &\sim& -\log x,  \quad x \downarrow 0, \\
\label{Ktendinfinity} K_{\nu} (x) &\sim& \sqrt{\frac{\pi}{2x}} \mathrm{e}^{-x}, \quad x \rightarrow \infty, \\
\label{roots} I_{\nu} (x) &\sim& \frac{\mathrm{e}^x}{\sqrt{2\pi x}}, \quad x \rightarrow \infty.
\end{eqnarray}

\subsubsection{Differentiation}
\begin{eqnarray}\label{diffIii} I_0' (x) &=& I_1(x), \\
\label{diffKii} K_0'(x) &=& -K_1(x), \\
I_0''(x)\label{ponmu}&=&I_0(x)-\frac{I_{1}(x)}{x}, \\
\label{thmcor}K_0''(x)&=&K_0(x)+\frac{K_{1}(x)}{x}, \\
I_0^{(3)}(x) \label{mjyt}&=&-I_0(x)+\bigg(1+\frac{2}{x^2}\bigg)I_{1}(x), \\
\label{111213}K_0^{(3)}(x)&=&-K_0(x)-\bigg(1+\frac{2}{x^2}\bigg)K_{1}(x).
\end{eqnarray}

\subsubsection{Differential equation}
The modified Bessel differential equation is 
\begin{equation} \label{realfeel} x^2 f''(x) + xf'(x) - (x^2 +\nu^2)f(x) =0.\end{equation}
The general solution is $f(x)=AI_{\nu} (x) +BK_{\nu} (x).$

\section{Bounds for expressions involving derivatives and integrals of modified Bessel functions}
The following bounds, which can be found in Gaunt \cite{gaunt confirmation}, \cite{gaunt}, are used to bound the derivatives of the solution to the $\mathrm{\mathrm{PN}}(2,\sigma^2)$ Stein equation (\ref{deltappp}).  For $x\geq 0$,
\begin{eqnarray}
 \left|I_{(0,0,n)}(x)  K_0^{(n)}(x)\right| &\leq& 2^{n-1}, \qquad n=1,2,3,\ldots, \nonumber  \\
 \left|xI_{(0,0,n)}(x)  K_0^{(n)}(x)\right| &\leq& 2^{n-1}, \qquad n=1,2,3,\ldots, \nonumber  \\
 \left|I_{(0,0,1)}(x)  K_0(x)\right| &<& 1, \nonumber \\
 \left|xI_{(0,0,1)}(x)  K_0(x)\right| &<& 1, \nonumber \\
\bigg|I_0(x)\int_x^{\infty}K_0(y)\,\mathrm{d}y\bigg|&\leq& \frac{\pi}{2}, \nonumber \\
\bigg|xI_0(x)\int_x^{\infty}K_0(y)\,\mathrm{d}y\bigg|&<& 0.615, \nonumber \\
\bigg|I_0^{(2n)}(x)\int_x^{\infty}K_0(y)\,\mathrm{d}y\bigg| &<& \frac{1}{4}+\frac{\sqrt{\pi}}{2}, \qquad n=0,1,2,\ldots, \nonumber \\
\bigg|I_0^{(2n+1)}(x)\int_x^{\infty}K_0(y)\,\mathrm{d}y\bigg| &<& \frac{1}{2}, \qquad n=0,1,2,\ldots, \nonumber \\
\bigg|xI_0'(x)\int_x^{\infty}K_0(y)\,\mathrm{d}y\bigg|&\leq& \frac{1}{2}, \nonumber  \\
\bigg|\frac{1}{x}-I_{(0,0,1)}(x)K_0''(x)\bigg|&<& 3, \nonumber \\
\bigg|\frac{1}{x}+I_{(0,0,2)}(x)K_0^{(3)}(x)\bigg|&<& 5,  \nonumber \\
\bigg|\frac{1}{x}-I_{(0,0,3)}(x)K_0^{(4)}(x)\bigg|&<& 9, \nonumber   \\
\bigg|\frac{2}{x^2}+I_{(0,0,1)}(x)K_0^{(3)}(x)\bigg|&<& 4.39, \nonumber \\
\bigg|\frac{3}{x^2}-I_{(0,0,2)}(x)K_0^{(4)}(x)\bigg|&<& 6.81, \nonumber \\
\bigg|\frac{6}{x^3}+\frac{1}{x}-I_{(0,0,1)}(x)K_0^{(4)}(x)\bigg|&<&14.61. \nonumber 
\end{eqnarray}

\section*{Acknowledgements}

During the course of this research the author was supported by an EPSRC DPhil Studentship and an EPSRC Doctoral Prize, and is currently supported by EPSRC grant EP/K032402/1.  The author would like to thank Gesine Reinert for the valuable guidance she provided on this project.  The author would also like to thank Larry Goldstein for a helpful discussion, which lead to section of this paper concerning the connection between the zero bias transformation of order $n$ and the square bias transformation.  Finally, the author would like to thank the anonymous referees and an associate editor for some highly constructive comments and suggestions. The author is particularly grateful to one referee, who gave some interesting insights regarding the zero bias transformation of order $n$, which lead to a substantial improvement in the presentation of this paper.


\begin{thebibliography}{99}

\bibitem{aaps16} Arras, B., Azmoodeh, E., Poly, G. and Swan, Y.  Stein's method on the second Wiener chaos : 2-Wasserstein distance. arXiv:1601:03301, 2016.

\bibitem{azmooden} Azmooden, E., Peccati, G. and Poly, G.  Convergence towards linear combinations of chi-squared random variables: a Malliavin-based approach.  \emph{S\'{e}minaire de Probabilit\'{e}s} XLVII (special volume in memory of Marc Yor) (2015), pp. 339--367.

\bibitem{baldi} Baldi, P., Rinott, Y. and Stein, C.  A normal approximation for the number of local maxima of a random function on a graph. In \emph{Probability, Statistics and Mathematics, Papers in Honor of Samuel Karlin} (T. W. Anderson, K. B. Athreya and D. L. Iglehart, eds.) (1989), pp. 59--81. Academic Press, New York.

\bibitem{chatterjee} Chatterjee, S., Fulman, J. and  R\"ollin, A. Exponential approximation by Stein's method and spectral graph theory.
\emph{ALEA Lat. Am. J. Probab. Math. Stat.} $\mathbf{8}$ (2011), pp. 197--223.

\bibitem{chen 0} Chen, L. H. Y.  Poisson approximation for dependent trials.  \emph{Ann. Probab.} $\mathbf{3}$ (1975), pp. 534--545.

\bibitem{chen} Chen, L. H. Y., Goldstein, L. and Shao, Q--M.  \emph{Normal Approximation by Stein's Method.} Springer, 2011.

\bibitem{collins} Collins, P. J.  \emph{Differential and Integral Equations.}  Oxford University Press, 2006.

\bibitem{dobler1} D\"{o}bler, C. Stein's method of exchangeable pairs for the beta distribution and generalizations. \emph{Electron. J. Probab.} no. 109 (2015), pp. 1--34.

\bibitem{dobler} D\"{o}bler, C. Distributional transformations without orthogonality relations. To appear in \emph{J. Theoret. Probab.}, 2016+.

\bibitem{dgv} D\"{o}bler, C., Gaunt, R. E. and Vollmer, S. J.  An iterative technique for bounding derivatives of solutions of Stein equations. arXiv:1510:02623, 2015.

\bibitem{eden1} Eden, R. and Viens, F. General upper and lower tail estimates using Malliavin calculus and  Stein's equations. In: \emph{Seminar on Stochastic Analysis, Random Fields and Applications VII}, Eds: Dalang, R. C., Dozzi, M. and Russo, F.,  Progress in Probability 67, Birkhauser, 2013.

\bibitem{eden2} Eden, R. and Viquez, J. Nourdin-Peccati analysis on Wiener and Wiener-Poisson space for general distributions. \emph{Stoch. Proc. Appl.} $\mathbf{125}$ (2015), pp. 182--216.

\bibitem{eichelsbacher} Eichelsbacher, P. and Th\"{a}le, C.  Malliavin-Stein method for Variance-Gamma approximation on Wiener space.  	\emph{Electron. J. Probab.} $\mathbf{20}$ no. 123 (2015), pp. 1--28.

\bibitem{folland} Folland, G.  \emph{Real Analysis: Modern Techniques and Their Applications.}  Wiley, New York, 1984.

\bibitem{gaunt confirmation} Gaunt, R. E. \emph{Rates of Convergence of Variance-Gamma Approximations via Stein's Method.}
DPhil thesis, University of Oxford, 2013.

\bibitem{gaunt vg} Gaunt, R. E.  Variance-Gamma approximation via Stein's method.  \emph{Electron. J. Probab.} $\mathbf{19}$ No. 38 (2014), pp. 1--33.

\bibitem{gaunt ngb} Gaunt, R. E.  Products of normal, beta and gamma random variables: Stein operators and distributional theory.  arXiv:1507.07696, 2015.

\bibitem{gaunt normal} Gaunt, R. E. Stein's method for functions of multivariate normal random variables. arXiv:1507.08688, 2015.

\bibitem{gaunt} Gaunt, R. E.  Uniform bounds for expressions involving modified Bessel functions.  To appear in \emph{Math. Inequal. Appl.}, 2016+.

\bibitem{gaunt gh} Gaunt, R. E.  A Stein characterisation of the generalized hyperbolic distribution. arXiv:1603:05675, 2016.

\bibitem{gaunt chi square} Gaunt, R. E., Pickett, A. and Reinert, G.  Chi-square approximation by Stein's method with application to Pearson's statistic.  arXiv:1507.01707, 2015.

\bibitem{goldstein} Goldstein, L. and Reinert, G.  Stein's Method and the zero bias transformation with application to simple random sampling.  \emph{Ann. Appl. Probab.} $\mathbf{7}$ (1997), pp. 935--952.

\bibitem{goldstein 2} Goldstein, L. and Reinert, G.  Zero biasing in one and higher dimensions, and applications. In \emph{Stein's Method and Applications. Lect. Notes Ser. Inst. Math. Sci. Natl. Univ. Singap.} $\mathbf{5}$  (2005), pp. 1–-18, Singapore Univ. Press, Singapore.

\bibitem{goldstein3} Goldstein, L. and Reinert, G.  Distributional transformations, orthogonal polynomials, and Stein characterisations.  \emph{J. Theoret. Probab.} $\mathbf{18}$ (2005), pp. 237--260.

\bibitem{goldstein4} Goldstein, L. and Reinert, G.  Stein's method for the Beta distribution and the P\'{o}lya-Eggenberger Urn.  \emph{J. Appl. Probab.} $\mathbf{50}$ (2013), pp. 1187--1205.

\bibitem{goldstein1} Goldstein, L. and Rinott, Y.  Multivariate normal approximations by Stein's method and size
bias couplings.  \emph{J. Appl. Probab.} $\mathbf{33}$ (1996),  1--17. 

\bibitem{kreyszig} Kreyszig, E. \emph{Advanced Engineering Mathematics,} 8th ed. Wiley, 1999.

\bibitem{ley} Ley, C., Reinert, G. and Swan, Y.  Approximate computation of expectations : a canonical Stein operator. arXiv:1408.2998, 2014.

\bibitem{luk} Luk, H.  \emph{Stein's Method for the Gamma Distribution and Related Statistical Applications.}  PhD thesis, University of Southern California, 1994. 

\bibitem{luke} Luke, Y. L. \emph{The Special Functions and their Approximations, Vol. 1}, Academic Press, New York, 1969.

\bibitem{nourdin1} Nourdin, I. and Peccati, G. Stein's method on Wiener chaos. \emph{Probab. Theory Rel.} $\mathbf{145}$ (2011), pp. 75--118.

\bibitem{np12} Nourdin, I. and Peccati, G.  \emph{Normal Approximations with Malliavin Calculus: From Stein's Method to Universality}. Cambridge Tracts in Mathematics. Cambridge University Press, 2012.

\bibitem{olver} Olver, F. W. J., Lozier, D. W., Boisvert, R. F. and Clark, C. W.  \emph{NIST Handbook of Mathematical
Functions.} Cambridge University Press, 2010.

\bibitem{pekoz1} Pek\"oz, E. and R\"ollin, A. New rates for exponential approximation and the theorems of R\'{e}nyi and Yaglom. \emph{Ann. Probab.} $\mathbf{39}$ (2011), pp. 587--608.

\bibitem{pekoz} Pek\"oz, E., R\"ollin, A. and Ross, N. Degree asymptotics with rates for preferential attachment random graphs. \emph{Ann. Appl. Prob.} $\mathbf{23}$ (2013), pp. 1188--1218.

\bibitem{pike} Pike, J. and Ren, H. Stein's method and the Laplace distribution. \emph{ALEA Lat. Am. J. Probab. Math. Stat.} $\mathbf{11}$ (2014), pp. 571-587.

\bibitem{reinert 0}  Reinert, G.  Three general approaches to Stein's method. In \emph{An Introduction
to Stein's Method. Lect. Notes Ser. Inst. Math. Sci. Natl. Univ. Singap.} Eds: Barbour, A. D. and Chen L. H. Y., $\mathbf{4}$ (2005), pp. 183–-221.
Singapore Univ. Press, Singapore.

\bibitem{springer} Springer, M. D. and Thompson, W. E.  The distribution of products of Beta, Gamma and Gaussian random variables.  \emph{SIAM J. Appl. Math.} $\mathbf{18}$ (1970), pp. 721--737.

\bibitem{stuart} Stuart, A. and Ord, J. K. \emph{Kendall's Advanced Theory of Statistics: Vol. I,} fifth ed. Charles Griffin, 1987.

\bibitem{stein} Stein, C.  A bound for the error in the normal approximation to the the distribution of a sum of dependent random variables.  In \emph{Proc. Sixth Berkeley Symp. Math. Statis. Prob.} (1972), vol. 2, Univ. California Press, Berkeley, pp. 583--602.

\bibitem{stein2}Stein, C.  \emph{Approximate Computation of Expectations.} IMS, Hayward, California, 1986.

\bibitem{stein3} Stein, C., Diaconis, P., Holmes, S. and Reinert, G.  Use of exchangeable pairs in the analysis of simulations. In \emph{Stein's method: Expository lectures and applications}, volume 46 of IMS Lecture Notes Monogr. Ser., Inst. Math. Statist., Beachwood, OH.  Eds: Diaconis, P. and Holmes, S.,  (2004), pp. 1--26.

\end{thebibliography}
\end{document}